\documentclass{amsart}
\usepackage{latexsym}
\usepackage{amsthm}
\usepackage{amsmath}
\usepackage{amsfonts}
\usepackage{amssymb}
\usepackage[dvips]{graphicx}
\usepackage[usenames,dvipsnames]{color}

\input xy
\xyoption{all}
\newdir{ >}{{}*!/-7pt/\dir{>}} 
\SelectTips{cm}{}

\theoremstyle{plain}
\newtheorem*{theoa}{Theorem 1}
\newtheorem*{theob}{Theorem 2}
\newtheorem*{theoc}{Theorem 3}
\newtheorem*{theod}{Theorem 4}
\newtheorem{theo}{Theorem}[section]
\newtheorem{lemma}[theo]{Lemma}
\newtheorem{propo}[theo]{Proposition}
\newtheorem{coro}[theo]{Corollary}

\theoremstyle{definition}
\newtheorem{defi}[theo]{Definition}

\theoremstyle{remark}
\newtheorem{rem}[theo]{Remark}

\newcommand\Map{\operatorname{Map}}

\newcommand\id{\operatorname{id}}

\newcommand\Aut{\operatorname{Aut}}

\newcommand\hcN{\mathit{hcN}}

\newcommand\sets{\mathcal{S}\mathrm{ets}}
\newcommand\ssets{\mathrm{s}\mathcal{S}\mathrm{ets}}
\newcommand\dsets{\mathrm{d}\mathcal{S}\mathrm{ets}}
\newcommand\sOper{\mathrm{s}\mathcal{O}\mathrm{per}}

\newcommand\sCat{\mathrm{s}\mathcal{C}\mathrm{at}}
\newcommand\Cat{\mathcal{C}\mathrm{at}}
\newcommand\Oper{\mathcal{O}\mathrm{per}}
\newcommand\Alg{\mathcal{A}\mathrm{lg}}
\newcommand\Idem{\mathrm{Idem}}
\newcommand\Split{\mathrm{Split}}

\newcommand\Morita{\mathrm{Mor}}

\newcommand\Iso{\mathrm{Iso}}
\newcommand\Fun{\mathrm{Fun}}

\newcommand{\myrightleftarrows}[1]{\mathrel{\substack{\xrightarrow{#1} \\[-.8ex] \xleftarrow{#1}}}}

\newcommand{\morp}[3]{#1 \colon #2 \to #3}

\newcommand{\trf}[1]{{#1}^u}

\newcommand{\Ret}{\mathrm{Ret}}

\newcommand{\St}{\mathfrak{C}}
\newcommand{\Ug}{U}
\newcommand{\cell}[1]{{#1}\text{-}\mathrm{cell}}

\newcommand{\inj}[1]{{#1}\text{-}\mathrm{inj}}

\newcommand{\cat}{\mathcal}

\newcommand{\str}{\mathbf}

\newcommand{\op}{^{\mathrm{op}}}

\newcommand{\N}{\mathbb{N}}
\newcommand{\Fin}{\mathrm{Fin}}
\newcommand{\Sym}{\mathbf{\Sigma}}
\newcommand{\Sign}[1]{{\mathrm{Sign}(#1)}}

\newcommand{\ope}{\mathcal}

\newcommand{\Mcc}{\cat{M}}

\newcommand{\CM}{\mathrm{CM}}

\numberwithin{equation}{section}

\addtocontents{toc}{\protect\setcounter{tocdepth}{1}}

\begin{document}
\title[Morita homotopy theory for $(\infty,1)$-categories and $\infty$-operads]{Morita homotopy theory for\\ $(\infty,1)$-categories and $\infty$-operads}

\author[G. Caviglia]{Giovanni Caviglia}
\address{Radboud Universiteit Nijmegen, Institute for
Mathematics, Astrophysics, and Particle Physics, Heyendaalseweg 135, 6525 AJ
Nijmegen, The Netherlands}
\email{giova.caviglia@gmail.com}

\author[J.\,J. Guti\'errez]{Javier J.~Guti\'errez}
\address{Departament de Matem\`atiques i Inform\`atica,
Facultat de Matem\`atiques i Infor\-m\`a\-tica,
Universitat de Barcelona,
Gran Via de les Corts Catalanes 585,
08007 Barcelona, Spain}
\email{javier.gutierrez@ub.edu}
\urladdr{http://www.ub.edu/topologia/gutierrez}

\begin{abstract}
We prove the existence of Morita model structures on the categories of small simplicial categories, simplicial sets, simplicial operads and dendroidal sets, modelling the Morita homotopy theory of $(\infty,1)$-categories and $\infty$-operads. We give a characterization of the weak equivalences in terms of simplicial presheaves, simplicial algebras and slice categories. In the case of the Morita model structure for simplicial categories and simplicial operads, we also show that each of these model structures can be obtained as an explicit left Bousfield localization of the Bergner model structure on simplicial categories and the Cisinski--Moerdijk model structure on simplicial operads, respectively.
\end{abstract}

\maketitle
\section*{Introduction}
Morita theory describes an equivalence relation between objects with the same type of algebraic structure in terms of their representations. Classically, it was defined for associative and unitary rings. Two rings are Morita equivalent if their corresponding categories of (left) modules are equivalent. This relation can be generalized by replacing rings by small categories enriched in abelian groups~\cite{New72} (rings are precisely one object categories of such kind).

In the non-additive setting, that is, for small categories, module categories are replaced by presheaf categories. Thus, two small categories $\mathcal{C}$ and $\mathcal{D}$ are Morita equivalent if the associated presheaf categories $\widehat{\mathcal{C}}$ and $\widehat{\mathcal{D}}$ are equivalent. It is a well-known result in category theory that, for a functor $ f\colon\mathcal{C}\to\mathcal{D}$ between small categories, the induced adjunction between the presheaf categories
$$
f_!:\widehat{\mathcal{C}}\myrightleftarrows{\rule{0.4cm}{0cm}}\widehat{\mathcal{D}}: f^*
$$
is an equivalence of categories if and only if $f$ is fully faithful and essentially surjective up to retracts; see for instance~\cite{EZ76}, \cite{BD86}. In this case, the functor $f$ is called a \emph{Morita equivalence}. It is clear that every equivalence of categories is a Morita equivalence, but the converse is not true in general. The difference is somehow measured by Cauchy completion. Every Morita equivalence between categories in which every idempotent splits is an equivalence of categories.

This notion of Morita equivalence can be also extended to the enriched case, where the category of sets is replaced by a monoidal category $\mathcal{V}$, small categories are enriched over $\mathcal{V}$ and presheaf categories are replaced by categories of $\mathcal{V}$-enriched functors~\cite{Lin74}.

From the homotopical point of view, it is a natural question to ask if Morita equivalences correspond to the weak equivalences of certain model structures, in the sense of Quillen, for the category of algebraic objects under consideration. The answer to this question has been studied in the literature in several contexts. Dell'Ambrogio and Tabuada proved in \cite[Theorem 1.1]{dAT1} that  if $R$ is a commutative ring, then there exists a model structure on the category of small $R$-categories whose weak equivalences are the Morita equivalences. In another paper, they studied Morita homotopy theory of $C^*$-categories extending the classical notion of Morita equivalence and proved the existence of a Morita model structure on the category of small unital $C^*$-categories~\cite[Theorem 4.9]{dAT2}. 

Another relevant example is the case of DG-categories, that is, categories enriched over chain complexes. Motivated by the study of homological invariants of DG-categories, Tabuada proved in \cite[Th\' eor\`eme 5.3]{Tab05} that the category of small DG-categories admits a model structure whose weak equivalences are the DG\nobreakdash-functors that induce an equivalence  between the corresponding derived categories, or equivalently, the functors that are locally quasi-isomorphisms and essentially surjective after taking the idempotent completion of the pretriangulated closure.

The aim of this paper is to develop a Morita homotopy theory for $(\infty,1)$\nobreakdash-cat\-egories and $\infty$-operads. There are different approaches to the theory of higher categories and higher operads in terms of model structures. In this paper, we will use the Bergner model structure on simplicial categories~\cite{Be07} and the Joyal model structure on simplicial sets~\cite{Joy08}, \cite{Lu09} to model $(\infty,1)$-categories,  and the  Cisinski--Moerdijk model structure on simplicial operads~\cite{CM13} and the operadic model structure on dendroidal sets~\cite{CM11} to model $\infty$-operads.

The main results of the paper are the existence of Morita model structures on the above categories, modelling the Morita homotopy theory for $(\infty,1)$-categories and $\infty$-operads, as we now summarize with more detail.

We call a map of simplicial categories a \emph{Morita weak equivalence} if it is homotopically fully faithful and homotopically essentially surjective up to retracts.

\begin{theoa}
There is a left proper cofibrantly generated model structure on small simplicial categories $\sCat_{\Morita}$ whose weak equivalences are the Morita weak equivalences. Moreover:
\begin{itemize}
\item[{\rm (i)}] A map of simplicial categories $f\colon \mathcal{C}\to \mathcal{D}$ is a Morita weak equivalence if and only if
the Quillen pair
$$
f_!:\ssets^{\mathcal{C}}\myrightleftarrows{\rule{0.4cm}{0cm}} \ssets^{\mathcal{D}}:f^*
$$
is a Quillen equivalence, where the categories of simplicial presheaves are equipped with the projective model structures.
\item[{\rm (ii)}] The model structure $\sCat_{\Morita}$ is a left Bousfield localization of the Bergner model structure on simplicial categories.
\end{itemize}
\end{theoa}
Our proof for the existence of the Morita model structure (Theorem~\ref{theo:moritamodelcat}) uses Kan's recognition principle for cofibrantly generated  model categories together with the existence of a generating set of \emph{retract intervals}, following the theory of enriched intervals described in~\cite{BM13}. Part (i) is Lemma~\ref{lem:Morita_r-equiv}, where we show that our definition of Morita weak equivalence coincides with the notion of weak $r$-equivalence, as introduced by Dwyer--Kan in~\cite{DK87}. Part (ii) is Corollary~\ref{cor:morita_scat_loc}.

We choose this approach to prove the existence of the model structure because it follows a similar strategy as in the case of the canonical model structure for simplicial categories, and also because it provides, a posteriori, more insight about the model structure itself. For instance, we get an explicit description of the generating cofibrations and trivial cofibrations (this is not the case if we construct the model structure directly as a left Bousfield localization).

We define the Morita model structure for quasicategories $\ssets_{\Morita}$ as the left Bousfield localization of the Joyal model structure on simplicial sets with respect to the morphism $N(\iota)\colon N(\Idem)\to N(\Split)$, where $N$ denotes the nerve functor and $\iota$ is the fully faithful functor characterizing the functors that lift split idempotents via right lifting property. Our main result in this setting is the following, which is proved in Theorem~\ref{thm:Quillen_eq_ssets_morita} and Corollary~\ref{cor:char_moritaeq_ssets}:

\begin{theob}The following hold for the Morita model structure on quasicategories:
\begin{itemize}
\item[{\rm (i)}] A map $f\colon X\to Y$ is a Morita weak equivalence of simplicial sets if and only if the adjunction
$$
f_!:\ssets/X\myrightleftarrows{\rule{0.4cm}{0cm}} \ssets/Y:f^*
$$
is a Quillen equivalence between the slice categories with the covariant model structures.
\item[{\rm (ii)}] There is a Quillen equivalence
$$
\St:\ssets_{\Morita}\myrightleftarrows{\rule{0.4cm}{0cm}} \sCat_{\Morita}: \hcN,
$$
where $\hcN$ is the homotopy coherent nerve and $\St$ is its left adjoint.
\end{itemize}
\end{theob}

For simplicial operads, we define Morita weak equivalences as a generalization of those of simplicial categories, that is, they are the homotopically fully faithful maps whose underlying functor of categories is essentially surjective up to retracts. To the authors knowledge this notion of Morita equivalence for operads has not been considered previously in the literature. We prove the following result in Proposition~\ref{pro:morita_soper_loc}, Theorem~\ref{thm:Morita_simp_oper} and Theorem~\ref{theo:main morita oper}:
\begin{theoc}
There is a cofibrantly generated model structure on simplicial operads $\sOper_{\Morita}$ whose weak equivalences are
the Morita weak equivalences of simplicial operads. Moreover:
\begin{itemize}
\item[{\rm(i)}] A morphism $f\colon\mathcal{O}\to \mathcal{P}$ between $\Sigma$-cofibrant simplicial operads is a Morita weak equivalence if and only if the Quillen pair
$$
f_!:\Alg(\ope{O})\myrightleftarrows{\rule{0.4cm}{0cm}} \Alg(\ope{P}): f^*
$$
is a Quillen equivalence, where the categories of algebras are equipped with the (transferred) projective model structure.
\item[{\rm (ii)}] The model structure $\sOper_{\Morita}$ is a left Bousfield localization of the Ci\-sins\-ki--Moerdijk model structure on simplicial operads.
\end{itemize}
\end{theoc}

Since every cofibrant simplicial operad in the Cisinski--Moerdijk model structure is $\Sigma$\nobreakdash-cofibrant and the cofibrant resolution of every operad provides a model for the corresponding notion of homotopy invariant algebraic structure, the Morita model structure on simplicial operads provides a model for a homotopy theory of homotopy invariant algebraic structures.

The characterization of the Morita weak equivalences in terms of categories of algebras given in part (i) is more involved than in the case of simplicial categories, since we have to deal with multi-linear algebraic structures. In order to handle this problem, we make use of multi-sorted simplicial algebraic theories and Morita equivalences in that context, and its relationship with simplicial operads, which we studied in~\cite{CG18}.

In the case of dendroidal sets, we define the Morita model structure $\dsets_{\Morita}$ as the left Bousfield localization of the operadic model structure with respect to the morphism $N_dj_!(\iota)$, where $N_d$ is the dendroidal nerve functor and $j_!$ is the left adjoint of the functor that sends an operad to its underlying category. Our main result in this setting is the following, which is proved in Theorem~\ref{thm:Quillen_eq_dsets_morita} and Corollary~\ref{cor:char_moritaweq_dsets}:

\begin{theod}
The following hold for the Morita model structure on dendroidal sets:
\begin{itemize}
\item[{\rm (i)}] A map between normal dendroidal sets $f\colon X\to Y$ is a Morita weak equivalence if and only if the adjunction
$$
f_!:\dsets/X\myrightleftarrows{\rule{0.4cm}{0cm}} \dsets/Y:f^*
$$
is a Quillen equivalence between the slice categories with the covariant model structures.
\item[{\rm (ii)}] There is a Quillen equivalence
$$
\St_d:\dsets_{\Morita}\myrightleftarrows{\rule{0.4cm}{0cm}} \sOper_{\Morita}: \hcN_d,
$$
where $\hcN_d$ is the homotopy coherent dendroidal nerve and $\St_d$ is its left adjoint.
\end{itemize}
\end{theod}

\bigskip
\noindent {\bf Organization of the paper.} In Section 1, we prove the existence of the Morita model structure on small categories. In Section 2, we define the Morita weak equivalences of simplicial categories and prove the existence of the Morita model structure, characterizing the weak equivalences in terms of categories of simplicial presheaves. In Section 3 we define the Morita model structure for quasicategories as a left Bousfield localization of the Joyal model structure and characterize the weak equivalences in terms of slice categories. In Section 4, we prove the existence of the Morita model structure on operads. In Section 5, we define the Morita weak equivalences of simplicial operads and prove the existence of the Morita model structure. We then use the results of~\cite{CG18} to characterize the Morita weak equivalences in terms of categories of algebras. Finally, in Section 6 we define the Morita model structure on dendroidal sets as a left Bousfield localization of the Cisinski--Moerdijk model structure and characterize the weak equivalences in terms of slice categories.

\bigskip
\noindent {\bf Acknowledgements.} We would like to thank Joost Nuiten for several useful conversations related to the subject of this paper. We would also like to thank the referee for very useful comments and suggestions that helped improving the
presentation of the paper.

The second named author was supported by the Spanish Ministry of Economy under the grants MTM2016-76453-C2-2-P (AEI/FEDER, UE) and RYC-2014-15328 (Ram\'on y Cajal Program).

\section{The Morita model structure for categories}
\label{sect:Morita_cat}
We begin this section by recalling the canonical model structure on small categories~\cite{JT91}, \cite{Re00}, where the weak equivalences are the categorical equivalences and the fibrations are the functors that lift isomorphisms. After that, we introduce the notion of Morita equivalence of small categories. These are the functors that are fully faithful and essentially surjective up to retracts. We prove that there is a Morita model structure on small categories, obtained as a left Bousfield localization of the canonical model structure, whose weak equivalences are the Morita equivalences.

\subsection{The canonical model structure}\label{subsect:can_cat}

An \emph{isofibration} of categories is a functor that lifts isomorphisms,  that is, a functor that has the right lifting property with respect to the inclusion $0\to J$, where $0$ is the category with one object and only the identity morphism, and $J$ is the category with two objects $0$ and $1$ and one isomorphism between them.

The category of small categories admits a cofibrantly generated proper model structure, called the \emph{canonical model structure}, in which weak equivalences are the categorical equivalences, that is, fully faithful and essentially surjective functors; fibrations are the isofibrations; and cofibrations are the functors that are injective on objects; see~\cite[Theorem 4]{JT91}, \cite[Theorem 3.1]{Re00}. A set of generating trivial cofibrations consists of the map $0\to J$ and a set of generating cofibrations consists of the maps $\emptyset\to 0$, $0\coprod 0\to I$ and $P\to I$, where $I$ denotes the category with two objects and exactly one non-identity map between them, and $P$ denotes the category with two objects and exactly two parallel arrows (the morphism $P\to I$ is the obvious map sending the two parallel arrows to the non-identity morphism).

The canonical model structure on $\Cat$ is a simplicial model structure\cite[Theorem 6.2]{Re00}. Given two categories $\mathcal{C}$ and $\mathcal{D}$, the simplicial enrichment is defined as $\Map(\mathcal{C}, \mathcal{D})=N(\Iso(\Fun(\mathcal{C},\mathcal{D})))$, where $\Fun(\mathcal{C},{\mathcal{D}})$ denotes the category of functors from $\mathcal{C}$ to $\mathcal{D}$ and $N$ is the nerve functor from small categories to simplicial sets. Given a category $\mathcal{C}$ we denote by $\Iso(\mathcal{C})$ the maximal subgroupoid of $\mathcal{C}$, that is, $\Iso(\mathcal{C})$ has the same objects as $\mathcal{C}$ and the morphisms are the isomorphisms.

\subsection{The Morita model structure}\label{subsect:Morita_cat}

We say that a functor $f\colon \mathcal{C}\to \mathcal{D}$ between small categories is \emph{essentially surjective up to retracts} if every object in $\mathcal{D}$ is isomorphic to a retract of an object in the image of $f$. 
\begin{defi}
A functor between small categories is called a \emph{Morita equivalence} if it is fully faithful and essentially surjective up to retracts.
\end{defi}

Let $\Idem$ be the category freely generated by one object $0$ and one non-identity arrow $e$ such that $e\circ e=e$, and let $\Split$ be the category freely generated by two objects $0$ and $1$ and two non-identity arrows $r\colon 0\to 1$ and $i\colon 1\to 0$ such that $r\circ i={\rm id}$.
We denote by $\iota\colon\Idem\to\Split$ the fully faithful functor that sends $0$ to $0$ and $e$ to $i\circ r$. A functor between small categories \emph{lifts split idempotents} if it has the right lifting property with respect to $\iota$. The functor $\iota$ will play an important role throughout the paper.

An \emph{idempotent} in a category $\mathcal{C}$ is a morphism $e\colon x\to x$ such that $e\circ e=e$. An idempotent $e$ \emph{splits} if $e=i\circ r$, with $i\colon y\to x$, $r\colon x\to y$ and $r\circ i={\rm id}$.
If $e$ splits, then the splitting is unique (up to unique isomorphism) since $i$ is the equalizer and $r$ is the coequalizer of the diagram $e, {\rm id}\colon x\rightrightarrows x$, respectively.

A category in which every idempotent splits is called \emph{Cauchy complete} or \emph{Karoubi complete}. Observe that a category $\mathcal{C}$ is Cauchy complete if and only if the map $\mathcal{C}\to 0$ has the right lifting property with respect to $\iota\colon{\rm Idem}\to {\rm Split}$.

There is an explicit construction of the Cauchy completion $\overline{\mathcal{C}}$ of a category.
The objects of $\overline{\mathcal{C}}$ are pairs $(x, e)$, where $x$ is an object of $\mathcal{C}$ and
$e\colon x\to x$ is an idempotent. A morphism $g\colon(x, e) \to (x', e')$ in $\overline{\mathcal{C}}$ is a morphism $g\colon x\to x'$ such that $g\circ e = g = e'\circ g$. The canonical functor $\mathcal{C}\to \overline{\mathcal{C}}$ that sends $x$ to $(x,{\rm id}_x)$ is a Morita equivalence.

Given a category $\mathcal{C}$, we denote by $\widehat{\mathcal{C}}$ its category of presheaves, that is, the category of functors $\mathcal{C}^{\rm op}\to \sets$. A classical result in category theory states that the following are equivalent for a functor $f\colon\mathcal{C}\to \mathcal{D}$ between small categories (see~\cite[Theorem 3.6']{EZ76}, \cite[Theorem 1]{BD86}):
\begin{itemize}
\item[{\rm (i)}] $f$ is a Morita equivalence.
\item[{\rm (ii)}] The left Kan extension $f_!:\widehat{\mathcal{C}}\rightleftarrows\widehat{\mathcal{D}}: f^*$ is an equivalence of categories.
\item[{\rm (iii)}] $\overline{f}\colon \overline{\mathcal{C}}\to\overline{\mathcal{D}}$ is an equivalence of categories.
\end{itemize}
Note also that a functor $f$ is a Morita equivalence if and only if $f^{\rm op}$ is a Morita equivalence.

\begin{theo}\label{thm:Morita_cat}
There is a cofibrantly generated model structure $\Cat_{\Morita}$ on the category of small categories in which the weak equivalences are the Morita equivalences and the cofibrations are the functors that are injective on objects. The fibrant objects are the Cauchy complete categories.
\end{theo}
\begin{proof}
The model structure $\Cat_{\Morita}$ is the left Bousfield localization of the canonical model structure on $\Cat$ with respect to the functor $\iota\colon{\rm Idem}\to {\rm Split}$. We just need to identify the fibrant objects and the weak equivalences of the local model structure.

The fibrant objects of $\Cat_{\Morita}$ are the $\iota$-local categories. It follows from the general theory of homotopy function complexes that if a category $\mathcal{C}$ is $\iota$-local, then $\mathcal{C}\to 0$ has the right lifting property with respect to $\iota$; see \cite[Corollary~17.7.5(2)]{Hi03}.

Conversely, let $\mathcal{C}$ be a Cauchy complete category. Then the induced functor
$$
\iota^*\colon \Iso(\Fun({\Split},\mathcal{D}))\longrightarrow \Iso(\Fun(\Idem, \mathcal{C}))
$$
is surjective on objects (in particular, essentially surjective). But $\iota^*$ is also fully faithful, because the splitting of idempotents is unique up to unique isomorphism (in fact, the functor $\Fun(\Split, \mathcal{C})\to \Fun(\Idem, \mathcal{C})$ is fully faithful if $\mathcal{C}$ is Cauchy complete). Thus $\iota^*$ is an equivalence of categories. Therefore
$$
N(\iota^*)\colon\Map(\Split, \mathcal{C})\longrightarrow \Map(\Idem,\mathcal{C})
$$
is a weak equivalence of simplicial sets and hence $\mathcal{C}$ is $\iota$-local.

The weak equivalences of $\Cat_{\Morita}$ are the $\iota$-local equivalences. Let
$f\colon \mathcal{C}\to\mathcal{D}$ be a Morita equivalence. For every category $\mathcal{A}$ the map
$f\times {\rm id}_{\mathcal{A}}$ induces an equivalence
of categories
$$
f^*\colon \Fun(\mathcal{D}, \sets^{\mathcal{A}{\op}})\longrightarrow\Fun(\mathcal{C},\sets^{\mathcal{A}{\op}}).
$$
The category $\Fun(\mathcal{C}, \mathcal{A})$ embeds in $\Fun(\mathcal{C},\sets^{\mathcal{A}{\op}})$. Therefore, if $\mathcal{A}$ is Cauchy complete, then
$$
f^*\colon \Fun(\mathcal{D}, \mathcal{A})\longrightarrow \Fun(\mathcal{C}, \mathcal{A})
$$
is fully faithful and essentially surjective. So, $f$ is an $\iota$-local weak equivalence.

Conversely, suppose that $f\colon \mathcal{C}\to \mathcal{D}$ is an $\iota$-local equivalence. Consider the following commutative diagram
$$
\xymatrix{
\mathcal{C}\ar[r]^f \ar[d] & \mathcal{D} \ar[d] \\
\overline{\mathcal{C}} \ar[r]_{\overline{f}} & \overline{\mathcal{D}},
}
$$
where ``overline'' denotes the Cauchy completion functor. The vertical maps are Morita equivalences, and therefore $\iota$-local weak equivalences by the argument in the previous paragraph. So $\overline{f}$ is an $\iota$-local weak equivalence, by the two out of three property for weak equivalences. Since $\overline{\mathcal{C}}$ and $\overline{\mathcal{D}}$ are $\iota$-local, $\overline{f}$ is an equivalence of categories and hence $f$ is a Morita equivalence.
\end{proof}

\section{The Morita model structure for simplicial categories}\label{sect:scat_morita}
In the first part of this section we recall some facts about simplicial categories, basically the description of the Bergner model structure on the category of small simplicial categories and its Quillen equivalence with the Joyal model structure on simplicial sets; see~\cite{Be07}, \cite{Joy07}, \cite[\S 2.2.5 and \S A.3.2]{Lu09} and \cite{DS11} for details.

In the second part, we prove the existence of the Morita model structure for simplicial categories and characterize the weak equivalences in terms of categories of simplicial presheaves. In order to achieve this, we need to develop a theory of \emph{retract intervals} similar to the theory of intervals used in the study of the homotopy theory of enriched categories~\cite{BM13}. Finally we show that the Morita model structure can be also obtained as an explicit left Bousfield localization of the Bergner model structure.

\subsection{Simplicial categories}
A \emph{small simplicial category} $\mathcal{C}$ is a small category  enriched in simplicial sets. For every two objects $x$ and $y$ of $\mathcal{C}$ we denote by $\mathcal{C}(x,y)$ the \emph{simplicial set of morphisms} from $x$ to $y$. For every three objects $x$, $y$ and $z$ in $\mathcal{C}$ there is a map of simplicial sets called the \emph{composition rule}
$$
 \mathcal{C}(y,z)\times \mathcal{C}(x,y)\longrightarrow \mathcal{C}(x,z);
$$
and for every object $x$ in $\mathcal{C}$, there is a map of simplicial sets $*\to \mathcal{C}(x,x)$, where $*$ is the terminal simplicial set, called the \emph{unit}. The composition rule is associative and compatible with the units.

A \emph{map of simplicial categories} $f\colon \mathcal{C}\to \mathcal{D}$ is a simplicial functor from $\mathcal{C}$ to $\mathcal{D}$, that is, a map of sets $f\colon{\rm Ob}(\mathcal{C})\to {\rm Ob}(\mathcal{D})$, where ${\rm Ob}(-)$ denotes the set of objects of the corresponding category, and maps of simplicial sets
$$
\mathcal{C}(x,y)\longrightarrow \mathcal{D}(f(x),f(y))
$$
compatible with the composition rule, for every two objects $x$ and $y$ of $\mathcal{C}$.

We will denote by $\sCat$ the \emph{category of small simplicial categories}. A simplicial category can be also viewed as a simplicial object in the category of small categories which has the same set of objects in every dimension, and the simplicial operators are the identity on the objects.

If $C$ is a set, we will denote by $\sCat_C$ the \emph{category of simplicial categories with $C$ as set of objects}. The morphisms are functors that induce the identity map on the objects.

\subsection{The Bergner model structure}
Let $\Cat$ denote the category of small categories. Then there is a functor $\pi_0\colon\sCat\to \Cat$ that sends a simplicial category~$\mathcal{C}$ to the category $\pi_0(\mathcal{C})$, called the \emph{path component category} of $\mathcal{C}$, which has the same objects as $\mathcal{C}$, and whose set of morphisms from $x$ to $y$ in $\pi_0(\mathcal{C})$ is $\pi_0(\mathcal{C}(x,y))$.

A  map of simplicial categories $f\colon \mathcal{C}\to \mathcal{D}$ is called \emph{homotopically fully faithful} or \emph{local weak equivalence} (respectively a \emph{local fibration}) if the map
$$
\mathcal{C}(x,y)\longrightarrow \mathcal{D}(f(x),f(y))
$$
is a weak equivalence (respectively a fibration) of simplicial sets, for every two objects $x$ and $y$ of $\mathcal{C}$.  A map of simplicial categories $f$ is called \emph{essentially surjective} (respectively an \emph{isofibration}) if the functor $\pi_0(f)$ is essentially surjective (respectively an isofibration of categories). A map of simplicial categories that is homotopically fully faithful and essentially surjective is called a \emph{Dwyer--Kan equivalence}.

The category of simplicial categories with a fixed set of objects and the category of simplicial categories admit model structures with weak equivalences and fibrations described in terms of the previous maps~\cite[Proposition 7.2]{DK80}, \cite[Theorem~1.1]{Be07}.

\begin{theo}[Dwyer--Kan]
There is a cofibrantly generated simplicial proper model structure on the category of simplicial categories with a fixed set $C$ of objects $\sCat_C$, in which weak equivalences are the maps that are homotopically fully faithful and fibrations are the maps that are local fibrations.
\end{theo}

\begin{theo}[Bergner]
There is a cofibrantly generated proper model structure on the category of simplicial categories $\sCat$, called the Bergner model structure, in which weak equivalences are the Dwyer--Kan equivalences and fibrations are the maps that are local fibrations and isofibrations.
\end{theo}

\begin{rem}
Right properness for the Bergner model structure was proved in~\cite[Proposition 3.5]{Be07} while left properness was proved in~\cite[Proposition A.3.2.4]{Lu09} (see also \cite[Corollary 8.10]{CM13}).
\end{rem}
A set of generating cofibrations of the Bergner model structure can be described explicitly (see~\cite[Section 2]{Be07}). Given a simplicial set $X$ and $n\ge 0$, let $\Ug(X)$ denote the simplicial category with two objects $0$ and $1$ and whose only non-trivial simplicial set of morphisms is
$\Ug(X)(0,1)=X$. The set $\mathcal{I}$ of generating cofibrations of $\sCat$ consists of:
\begin{itemize}
\item[{\rm (i)}] The map $\emptyset\to 0$, where $0$ is the terminal simplicial category.
\item[{\rm (ii)}] The maps $\Ug(\partial\Delta[n])\to \Ug(\Delta[n])$ for $n\ge 0$.
\end{itemize}

In order to give a description of the set of generating trivial cofibrations, we need to recall first the notion of interval and generating set of intervals from~\cite{BM13}. Let $\mathbb{I}$ be the simplicial category with two objects $\{0,1\}$ representing a single isomorphism, that is, $\mathbb{I}(0, 0)= \mathbb{I}(1, 0) = \mathbb{I}(0, 1) = \mathbb{I}(1, 1) = *$.  Following the terminology of \cite[Definition 1.11]{BM13}, a simplicial category is called an $\emph{interval}$ if it is cofibrant in $\sCat_{\{0,1\}}$ and weakly equivalent to $\mathbb{I}$. A set $\mathcal{G}$ of intervals is called a \emph{generating set} if every interval is a retract of a trivial extension of an object in $\mathcal{G}$. More explicitly, if for every interval $\mathbb{H}$ there exist another interval $\mathbb{K}$, an interval $\mathbb{G}$ in $\mathcal{G}$ and a diagram in $\sCat_{\{0,1\}}$
$$
\xymatrix{
\mathbb{G} \ar[r] & \mathbb{K} \ar@/^/[r]^{r} & \ar@/^/[l]^{i}\mathbb{H},
}
$$
where $\mathbb{G}\to\mathbb{K}$ is a trivial cofibration in $\sCat_{\{0,1\}}$ and $r\circ i={\rm id}$. Since the model category of simplicial sets is combinatorial and monoidal,  \cite[Lemma 1.12]{BM13} implies that there exists a set of generating intervals in $\sCat$.

The set $\mathcal{J}$ of generating trivial cofibrations of $\sCat$ consists of:
\begin{itemize}
\item[{\rm (i)}] The maps $\Ug(\Lambda[n, k])\to \Ug(\Delta[n])$ for $n\ge 1$ and $0\le k\le n$.
\item[{\rm (ii)}] The inclusions $0\to \mathbb{G}$, where $\mathbb{G}$ runs over all the elements of a generating set of intervals $\mathcal{G}$.
\end{itemize}

The category of simplicial sets and the category of simplicial categories are related by a pair of adjoint functors, whose right adjoint is Cordier's homotopy coherent nerve. Given $n\ge 0$, let $[n]$ be the category freely generated by $n+1$ objects and just one non-identity arrow $i\to i+1$, for $0\le i< n$. Cordier described in~\cite{Cor82} how to associate to~$[n]$ a ``simplicial resolution'' $C_*[n]$ (cf. \cite[Definition 1.1.5.1]{Lu09}). This construction defines a functor $C_*\colon\Delta\to\sCat$ which, by left Kan extension, gives an adjunction
\begin{equation}\label{eq:adj_hcN}
\St:\ssets\myrightleftarrows{\rule{0.4cm}{0cm}} \sCat: \hcN,
\end{equation}
where the right adjoint is called the \emph{homotopy coherent nerve}. Given a simplicial category $\mathcal{C}$, the homotopy coherent nerve is defined by
$$
\hcN(\mathcal{C})_n=\sCat(C_*[n], \mathcal{C}).
$$
Joyal proved in \cite[Theorem 2.10]{Joy07} that  adjunction~\eqref{eq:adj_hcN} is a Quillen equivalence; see also~\cite[Theorem 2.2.5.1]{Lu09}, \cite[Theorem 1.5]{DS11} and \cite[Corollary 8.16]{CM13}.

\begin{theo}[Joyal, Lurie]
The adjunction $\St:\ssets\rightleftarrows\sCat:\hcN$ is a Quillen equivalence, where $\ssets$ is endowed with the Joyal model structure and $\sCat$ is endowed with the Bergner model structure.
\end{theo}

\subsection{Homotopy idempotents and retracts}

Let $\Ret$ be the simplicial set defined as the following pushout
\[
\xymatrix{\Delta[1] \ar[r]^-{d_1} \ar[d]_-{s_0}&
    \Delta[2] \ar[d]^-{} \\
    \Delta[0] \ar[r]_{} &
    \Ret.
}
\]
The maps of simplicial sets $1\colon \Delta[0]\to N(\Split)$ and $[i,r, r\circ i]\colon \Delta[2]\to N(\Split)$ induce a canonical map $\rho\colon \Ret\to N(\Split)$, where $\Split$ is the category defined in~Section \ref{subsect:Morita_cat} and $N$ is the nerve functor from categories to simplicial sets. Lurie proved in \cite[Proposition 4.4.5.6]{Lu09} that $\rho$ is an inner anodyne map of simplicial sets.
(Lurie uses the notation $\Idem^+$ for what we call $N(\Split)$.)

Let $\mathcal{C}$ be a fibrant simplicial category. Observe that a map of simplicial categories $a\colon {\St(\Delta[2])}\to {\mathcal{C}}$ is uniquely determined by three arrows $f\colon {x}\to {y}$,  $g\colon{y}\to {z}$, and $h\colon {x}\to {z}$
in $\mathcal{C}$ together with a homotopy ${\alpha}\colon {\Delta[1]}\to{\mathcal{C}(x,z)}$ from $g\circ f$ to $h$. Therefore, we will denote such a map $a$ by $[g,f,\alpha]$.

Note that the canonical map ${\sCat(\St(\Ret),\mathcal{C})}\to {\sCat(\St(\Delta[2]),\mathcal{C})}$ is injective. Indeed, $\St$ preserves pushouts, since it is a left adjoint, and an element $[g,f,\alpha]$ is in $\sCat(\St(\Ret),\mathcal{C})$ if and only if the domain of $f$ and the codomain of $g$ are the same object $y$ of $\mathcal{C}$ and $\alpha(1)=\id_y$.

Recall that, for a simplicial category $\mathcal{C}$, we denote by $\pi_0(\mathcal{C})$ its path component category. A functor $b\colon{\Split}\to{\pi_0(\mathcal{C})}$ is completely determined by two morphisms $[r]\in \pi_0(\mathcal{C}(x,y))$ and $[i]\in \pi_0(\mathcal{C}(y,x))$ such that $r\circ i$ is homotopic to $\id_y$ in $\mathcal{C}(y,y)$. Hence we will denote such a functor $b$ by $([r],[i])$.

\begin{lemma}
If $\mathcal{C}$ is a fibrant simplicial category, then the function
\[
{\sCat(\St(\Ret), \mathcal{C})}\to{\Cat(\Split, \pi_0(\mathcal{C}))}\]
that sends $[g, f, \alpha]$ to $([g], [f])$ is well defined and surjective.
\end{lemma}
\begin{proof}
The fact that the function is well defined is clear. To prove that it is surjective, let $([r],[i])$ be an element of $\Cat(\Split, \pi_0(\mathcal{C}))$. Let $x$ and $y$ be the domain and the codomain of $r$, respectively. Let ${\alpha}\colon {\Delta[1]}\to {\mathcal{C}(y,y)}$ be a homotopy between $r\circ i$ and $\id_{y}$.
The element $[r,i,\alpha]\in \sCat(\St(\Ret),\mathcal{C})$ is the desired lifting of $([r],[i])$.
\end{proof}

\subsection{The fibered model structure}

Let $f\colon A\to B$ be a map of sets. The map $f$ induces an adjunction between simplicial categories which have $A$ and $B$ as sets of objects
$$
f_!:\sCat_{A}\myrightleftarrows{\rule{0.4cm}{0cm}} \sCat_{B}:f^*.
$$
Thus, giving a map of simplicial categories $f\colon \mathcal{C}\to\mathcal{D}$ is the same as giving a map of sets $f\colon{\rm ob}(\mathcal{C})\to {\rm ob}(\mathcal{D})$ and a map $f^u\colon\mathcal{C}\to f^*(\mathcal{D})$ in $\sCat_{{\rm ob}(\mathcal{C})}$, or equivalently, a
map $f_u\colon f_!(\mathcal{C})\to \mathcal{D}$ in $\sCat_{{\rm ob}(\mathcal{D})}$.
There exists a model structure on $\sCat$, called the \emph{fibered model structure}, such that:
\begin{itemize}
    \item[{\rm (i)}] The weak equivalences are the local weak equivalences which are bijective on objects.
    \item[{\rm (ii)}] The fibration are the local fibrations.
    \item[{\rm (iii)}] The cofibrations are the maps of simplicial categories $f\colon\mathcal{C}\to\mathcal{D}$ such that $f_u\colon f_!(\mathcal{C})\to\mathcal{D}$ is a cofibration in $\sCat_{{\rm ob}(\mathcal{D})}$.
\end{itemize}
The fibered model structure on $\sCat$ was described by Joyal in~\cite[Section 4.1]{Joy07}. It can be also obtained as a particular instance of the \emph{integral model structure}~\cite[Theorem 3.0.12]{HP14}. We will denote the fibered model structure by $\sCat_{\rm fib}$.

\begin{rem}\label{rem:sCat_fib_rp}
The model structure $\sCat_{\rm fib}$ is right proper, since $\sCat_{\rm fib}$ has the same fibrations as $\sCat$, the weak equivalences in $\sCat_{\rm fib}$ are weak equivalences in $\sCat$, and $\sCat$ is right proper.
\end{rem}

\begin{lemma}\label{lem:injcof}
A map of simplicial categories is a cofibration in the fibered model structure and injective on objects if and only if it has the left lifting property with respect to trivial fibrations in the Bergner model structure (that is, with respect to local trivial fibrations surjective on objects).
\end{lemma}
\begin{proof}
By~\cite[Proposition 4.5]{Joy07} the cofibrations in the Bergner model structure are cofibrations in the fibered model structure for simplicial categories, and a cofibration in the fibered model structure is a cofibration in the Bergner model structure if and only if it is injective on objects.
\end{proof}

\subsection{Retract intervals}

Let $I$ be the category with two objects $0$ and $1$ and one non-identity arrow from $0$ to $1$. Let $\sCat^I$ denote the category of functors from $I$ to $\sCat$, that is, the category of arrows of $\sCat$. The category $I$ is a Reedy category and hence $\sCat^I$ (and also $\sCat^I_{\rm fib}$) admits a Reedy model structure; see, for instance \cite[Theorem 15.3.4]{Hi03}.

\begin{defi}
A \emph{retract interval} in $\sCat$ is a cofibrant object in $(\sCat_{\rm fib}^{I})_{{\rm Reedy}}$ weakly equivalent to $\iota\colon{\Idem}\to{\Split}$.
\end{defi}

In other words, a retract interval in $\sCat$ is a map $h\colon E\to R$ of simplicial categories  such that:
\begin{itemize}
    \item[{\rm (i)}] The category $E$ has one object $0$ and it is cofibrant in $\sCat_{\{0\}}$.
    \item[{\rm (ii)}] The category $R$ has two objects $0$ and $1$.
    \item[{\rm (iii)}] The map $h$ sends $0$ to $0$ and its associated map $h_u\colon {h_!E}\to {R}$ is a cofibration in $\sCat_{\{0,1\}}$.
    \item[{\rm (iv)}] There is a zigzag of weak equivalences in $(\sCat_{\rm fib}^{I})_{\rm Reedy}$\[
    \xymatrix{
        E \ar[r]^-{\sim} \ar[d]_-{h} &
        \Idem^f \ar[d]^-{\iota^f}&
        \Idem \ar[l]_-{\sim} \ar[d]^-{\iota}\\
        R \ar[r]_-{\sim} &
        \Split^f &
        \Split, \ar[l]^-{\sim}
    }
    \]
    where $\iota^f$ is a fibrant replacement for $\iota$.
\end{itemize}

\begin{rem}\label{rem:CN_retract}
    Any cofibrant replacement of $\iota$ is a retract interval. In particular, $\St N(\iota)$ is a retract interval
    \[
    \xymatrix{
        \St N(\Idem) \ar[r]^-{\sim} \ar[d]_-{\St N(\iota)} &
        \Idem \ar[d]^-{\iota} \\
        \St N(\Split) \ar[r]_-{\sim} &
        \Split.
    }
    \]
As proved in \cite[Lemma 1.20]{Joy07} or \cite[Theorem 2.2.0.1]{Lu09}, the horizontal maps in the diagram above are weak equivalences in $\sCat$.
\end{rem}

\begin{defi}
    A set $\mathcal{G}$ of retract intervals  in $\sCat$ is called a \emph{generating set} if every retract interval is a retract of a trivial extension of an element of $\mathcal{G}$ in $(\sCat_{\rm fib}^{I})_{{\rm Reedy}}$. That is, for every retract interval $h\colon E\to R$ there exists a commutative diagram
    \[
    \xymatrix{
        G_0 \ar@{ >->}[r]^-{\sim} \ar[d]_{g} &
        G_0' \ar@/^/[r]^{r} \ar[d]^{g'}&
        E \ar@/^/[l]^{i} \ar[d]^-{h}\\
        G_1 \ar@{ >->}[r]_-{\sim} &
        G_1' \ar@/^/[r]^{r} &
        R, \ar@/^/[l]^{i}
    }
    \]
where $r\circ i= {\rm id}$, $G_i\to G_i'$ are trivial cofibrations in $(\sCat_{\rm fib}^{I})_{{\rm Reedy}}$ for $i=0,1$, and~$g\in \mathcal{G}$.
\end{defi}

\begin{lemma}
There exists a generating set of retract intervals in~$\sCat$.
\end{lemma}
\begin{proof}
The existence of such a set can be proved along the same lines as in \cite[Lemma 1.12]{BM13}.  Let $I$ denote the category with two objects and exactly one non-identity map between them. The slice category $\sCat^I/\iota^{f}$ and the category ${(\sCat^I)}^I$ are locally presentable and the inclusion functor
$$
\Phi\colon\sCat^I/\iota^{f}\longrightarrow {(\sCat^I)}^I
$$
is accessible, since $\sCat^I/\iota^{f}$ is closed under filtered colimits in ${(\sCat^I)}^I$; see \cite[Example 2.17(3)]{AR}. Let $\mathcal{W}$ be the full subcategory of ${((\sCat_{\rm fib}^{I})_{{\rm Reedy}}})^I$ spanned by the weak equivalences in $(\sCat_{\rm fib}^{I})_{{\rm Reedy}}$. By~\cite[Corollary A.2.6.6]{Lu09}, $\mathcal{W}$ is an accessible and accessibly embedded subcategory. Thus, the preimage along the inclusion functor $\Phi^{-1}(\mathcal{W})$ is also an accessible and accessibly embedded subcategory and thus it has a set $\mathcal{G}'$ of objects that generate it under filtered colimits. But $\Phi^{-1}(\mathcal{W})$ is precisely the subcategory of $\sCat^I/\iota^{f}$ whose objects are of the form
$$
\xymatrix{
        E \ar[r]^-{\sim} \ar[d] &
        \Idem^f \ar[d]^-{\iota^f}\\
        R \ar[r]_-{\sim} &
        \Split^f.
        }
$$
Note that the map $E\to R$ is not a retract interval, since it is not cofibrant in general. Let $\mathcal{G}$ be the set that consists of taking a cofibrant replacement for every element in $\mathcal{G}'$. Then the elements of $\mathcal{G}$ are now retract intervals. It then follows that for every retract interval $h$ there is a retract interval $g$ in $\mathcal{G}$ and a weak equivalence $g\to h$. By using Ken Brown's lemma~\cite[Lemma 7.7.1]{Hi03}, we can factor $g\to h$  as a trivial cofibration $g\to g'$ followed by a retraction $g'\to h$ of a trivial cofibration $h\to g'$. This means precisely that $\mathcal{G}$ is a generating set of retract intervals.
\end{proof}

\begin{rem}\label{rem:rlp_interval}
Note that if a fibration in $\sCat_{{\rm fib}}$ has the right lifting property with respect to a generating set of retract intervals $\mathcal{G}$, then it has the right lifting property with respect to every retract interval.
\end{rem}

\begin{defi}
    Let $x$ and $y$ be two objects in a simplicial category $\mathcal{C}$. We say that
    \begin{itemize}
        \item[{\rm (i)}] $y$ is a \emph{strong homotopy retract} of $x$ if there exists a retract interval $h\colon E\to R$ and a functor $r\colon R\to \mathcal{C}$ such that $r(0)=x$ and $r(1)=y$;
        \item[{\rm (ii)}] $y$ is a \emph{virtual strong homotopy retract} of $x$ if it is a strong homotopy retract of $x$ in a fibrant replacement $\mathcal{C}^f$ of $\mathcal{C}$ in $\sCat_{\rm fib}$;
        \item[{\rm (iii)}] $y$ is a \emph{homotopy retract} of $x$ if it is a retract of $x$ in $\pi_0(\mathcal{C})$.
    \end{itemize}
\end{defi}

Our goal is to prove that the three notions defined above are, in fact, all equivalent. The arguments we give are similar to the ones used when dealing with intervals in the study of the homotopy theory of enriched categories (cf.~\cite[Lemma 2.10 and Lemma 2.4]{BM13}).

\begin{propo}\label{prop:strongret_equal_virtual}
    Let $\mathcal{C}$ be a simplicial category and let $x$ and $y$ be two objects of~$\mathcal{C}$. Then $y$ is a strong homotopy retract of $x$ if and only if $y$ is a virtual strong homotopy retract of $x$.
\end{propo}
\begin{proof}
The fact that if $y$ is a strong homotopy retract of $x$, then $y$ is a virtual strong homotopy retract of $x$ is straightforward.

Conversely, let $\mathcal{C}^f$ be a fibrant replacement of $\mathcal{C}$ in $\sCat_{\rm fib}$ and suppose that $y$ is a virtual homotopy retract of $x$. Then there exist a retract interval $h\colon {E}\to {R}$ and a map of simplicial categories $r\colon R\to \mathcal{C}^f$ such that $r(0)=x$ and $r(1)=y$.
The functor $r$ determines a unique morphism ${\overline{r}}\colon {h}\to {\id_{\mathcal{C}^f}}$ in $\sCat^I$.  Now we factor $\overline{r}$ into a trivial cofibration followed by a fibration
    \[
    \xymatrix{h \ar@{>->}[r]^-{\sim} &
        h' \ar@{->>}[r]^-{k} &
        \id_{\mathcal{C}^f}.}
    \]
Note that $h'$ is also a retract interval. Consider the following pullback square in $(\sCat_{{\rm fib}}^{I})_{{\rm Reedy}}$
    \[
    \xymatrix{
        h'' \ar[d]_-{\sim} \ar@{->>}[r]^-{k'}&
        \id_{X} \ar[d]^-{\sim}\\
        h' \ar@{->>}[r]_-{k}&
        \id_{X^f}.
    }
    \]
Since the model structure $\sCat_{\rm fib}$ is right proper (see Remark~\ref{rem:sCat_fib_rp}), the model structure $(\sCat_{\rm fib}^{I})_{{\rm Reedy}}$ is also right proper. Therefore, the vertical map on the left is a weak equivalence. In particular, $h''$ is a retract interval.

    Let $\morp{p}{\widetilde{h}}{h''}$ be a cofibrant replacement for $h''$. Then, the map $\widetilde{h}$ is a retract interval and the composition $k'\circ p\colon {\widetilde{h}}\to{\id_\mathcal{C}}$ exhibits $y$ as a strong homotopy retract of $x$.
\end{proof}

\begin{propo}\label{prop:virtual_equal_homret}
    Let $\mathcal{C}$ be a simplicial category and let $x$ and $y$ be two objects of~$\mathcal{C}$.
    Then $y$ is a virtual strong homotopy retract of $x$ if and only if $y$ is a homotopy retract of $x$.
\end{propo}
\begin{proof}
    One direction is clear.
    We are going to prove that if $y$ is a homotopy retract of $x$, then $y$ is a virtual strong homotopy retract of $x$. Let $\mathcal{C}^f$ be a fibrant replacement of $\mathcal{C}$ in $\sCat_{\rm fib}$. Note that since $\pi_0(\mathcal{C}^f)\cong \pi_0(\mathcal{C})$, the object $y$ is a homotopy retract of $x$ in $\mathcal{C}^f$ as well. Hence, there exists a map of simplicial categories $r\colon {\St(\Ret)}\to {\mathcal{C}^f}$ such that $r(0)=x$ and $r(1)=y$.

    By \cite[Proposition 4.4.5.6]{Lu09} the map $\Ret\to N(\Split)$ is an inner anodyne map of simplicial sets and thus a trivial cofibration in the Joyal model structure. Hence the map ${\St(\rho)}\colon{\St(\Ret)}\to{\St(N(\Split))}$ is a trivial cofibration in both $\sCat$ and $\sCat_{\rm fib}$, since $\St$ is a left Quillen functor. Therefore, $r$ lifts to a map of simplicial categories ${\tilde{r}}\colon {\St(N(\Split}))\to{\mathcal{C}^f}$. Since $\St(N({\Split}))$ is the target of a retract interval, namely $\St(N(\iota))$, this implies that $y$ is a virtual homotopy retract of $x$.
\end{proof}

\subsection{The Morita model structure}

In this section we prove the existence of the Morita model structure for simplicial categories. We begin by describing the weak equivalences and fibrations of this model structure.

\begin{defi}\label{def:ret_lift}
A map of simplicial categories $f\colon\mathcal{C}\to\mathcal{D}$ is called:
\begin{itemize}
\item[{\rm (i)}] \emph{Retract-lifting} if it has the right lifting property with respect to retract intervals.
\item[{\rm (ii)}] \emph{Homotopically essentially surjective up to retracts} if for every $a$ in $\mathcal{D}$ there exist $x$ in $\mathcal{C}$, a retract interval $h\colon E\to R$, and a map of simplicial categories $q\colon R\to \mathcal{D}$ such that $q(0)=f(x)$ and $q(1)=a$.
\end{itemize}
\end{defi}

\begin{defi}\label{def:Moreq}
    A map of simplicial categories is called a \emph{Morita weak equivalence} if it is homotopically fully faithful and homotopically essentially surjective up to retracts. A map of simplicial categories is called a \emph{Morita fibration} if it is a local fibration and retract-lifting.
\end{defi}

\begin{lemma}\label{lem:tfib}
A map is a Morita trivial fibration (that is, a Morita fibration which is also a Morita weak equivalence) if and only if it is a local trivial fibration which is surjective on objects.
\end{lemma}
\begin{proof}
By definition, every Morita trivial fibration is a local trivial fibration. So it suffices to prove that a local trivial fibration $f\colon\mathcal{C}\to\mathcal{D}$ is retract-lifting and homotopically essentially surjective up to retracts if and only if it is surjective on objects.

    Suppose that $f$ is retract-lifting and homotopically essentially surjective up to retracts and let $a$ be an object of $\mathcal{D}$. By hypothesis, there exist $x$ in $\mathcal{C}$, a retract interval $h\colon{E}\to{R}$ and a map of simplicial categories ${q}\colon{R}\to\mathcal{D}$ such that $q(0)=f(x)$ and $q(1)=a$.
Consider the following commutative diagram
    \[
    \xymatrix{0
        \ar[r]^-{x}
        \ar@{ >->}[d]_-{0} &
        \mathcal{C}\ar@{->>}[dd]^-{f}^-{}\\
        E \ar@{.>}[ur]^-{\tilde{x}}
        \ar[d]_-{h} & \\
        R \ar[r]_-{q} \ar@{.>}[uur]^-{p} & \mathcal{D}.
    }
    \]
We can decompose the map $x$ as a composition of a functor which is bijective on objects followed by a fully faithful one, and the same for the map $q\circ h$. Thus, we get a commutative diagram
$$
\xymatrix{
0\ar[r] \ar[d] & x^*\mathcal{C} \ar[d]\ar[r] & \mathcal{C} \ar[d] \\
E \ar[r]\ar@{.>}[ur]^{\tilde{x}} & h^*q^*\mathcal{D} \ar[r] & \mathcal{D}.
}
$$
The vertical map in the middle is a trivial fibration in the category of simplicial categories with one object and the map $0\to E$ is a cofibration which is bijective on objects. Therefore, $x$ lifts to a map $\tilde{x}$. Furthermore, since $f$ is retract-lifting, there exists a map $p$ that makes the first diagram commute. Clearly $f(p(1))=q(1)=a$. This proves that $f$ is surjective on objects.

Conversely, suppose that $f$ is surjective on objects. Then $f$ is clearly homotopically essentially surjective up to retracts, and it is also retract\nobreakdash-lifting by Lemma~\ref{lem:injcof}.  			
\end{proof}

\begin{rem}
The previous lemma shows that the class of trivial fibrations for the Morita model structure on simplicial categories is the same as the class of trivial fibrations for the Bergner model structure. In fact, as we will see in Section~\ref{sect:loc_scat}, the Morita model structure can be obtained as a left Bousfield localization of the Bergner model structure.
\end{rem}

Following Dwyer--Kan~\cite{DK87} a map $f$ of simplicial categories is called a \emph{weak $r$-equivalence} if $f$ is homotopically fully faithful and $\pi_0(f)$ is essentially surjective up to retracts.

Let $f\colon \mathcal{C}\to \mathcal{D}$ be a map of simplicial categories. Then there is a Quillen pair
$$
f_!:\ssets^{\mathcal{C}}\myrightleftarrows{\rule{0.4cm}{0cm}} \ssets^{\mathcal{D}}:f^*
$$
between the corresponding categories of simplicial functors with the projective model structure. Dwyer and Kan characterized the maps for which the previous adjunction is a Quillen equivalence; see~\cite[Theorem 2.1]{DK87} and \cite[IX. Theorem 2.13]{GJ99}.

\begin{theo}[Dwyer--Kan]\label{thm:DK}
Let  $f\colon \mathcal{C}\to \mathcal{D}$ be a map of simplicial categories. Then
$f_!:\ssets^{\mathcal{C}}\rightleftarrows \ssets^{\mathcal{D}}:f^*$ is a Quillen equivalence if and only if $f$ is a weak $r$-equivalence.
\end{theo}

\begin{lemma}\label{lem:Morita_r-equiv}
A map $f$ of simplicial categories is homotopically essentially surjective up to retracts if and only of $\pi_0(f)$ is essentially surjective up to retracts. In particular, a map of simplicial categories is a Morita weak equivalence if and only if it is a weak $r$-equivalence.
\end{lemma}
\begin{proof}
Suppose that $f\colon\mathcal{C}\to \mathcal{D}$ is homotopically essentially surjective up to retracts. Let $y$ be an object of $\mathcal{D}$. By assumption there exists an object $x$ in $\mathcal{C}$ such that $y$ is a strong homotopy retract of $f(x)$.
By Proposition~\ref{prop:strongret_equal_virtual} and Proposition~\ref{prop:virtual_equal_homret} this implies that $y$ is a homotopy retract of $f(x)$. It follows that $\pi_0(f)$ is essentially surjective up to retracts.
The converse is obvious.
\end{proof}

Recall that given a class $\mathcal{I}$ of morphisms in a category with all small colimits we denote by:
\begin{itemize}
\item[{\rm (i)}] $\mathcal{I}$-inj the class of morphisms with the right lifting property with respect to every element in $\mathcal{I}$.
\item[{\rm (ii)}] $\mathcal{I}$-cell the class of morphisms obtained as a transfinite composition of push\-outs of elements in $\mathcal{I}$.
\end{itemize}

\begin{theo}\label{theo:moritamodelcat}
    There is a cofibrantly generated model structure on $\sCat$, called the Morita model structure, and denoted by $\sCat_{\Morita}$ such that:
\begin{itemize}
        \item[{\rm (i)}] The weak equivalences are the Morita weak equivalences.
        \item[{\rm (ii)}] The fibrations are the Morita fibrations.
        \item[{\rm (iii)}] The cofibrations are the cofibrations of $\sCat$ with the Bergner model structure.
\end{itemize}
A set $\mathcal{J}_{\Morita}$ of generating trivial cofibrations consists of a generating set of retract intervals $\mathcal{G}$ together with the maps $\Ug(\Lambda[n, k])\to \Ug(\Delta[n])$ for $n\ge 1$ and $0\le k\le n$.
As a set of generating cofibrations we can take the set of generating cofibrations of the Bergner model structure.
\end{theo}
\begin{proof}
To prove the existence of the model structure we are going to use Kan's recognition principle for cofibrantly generated model categories; see \cite[Theorem 2.1.19]{Ho99} and \cite[Theorem 11.3.1]{Hi03}. It then suffices to check the following conditions:
    \begin{enumerate}
        \item\label{cond:modst1} The class of Morita weak equivalences $\mathcal{W}_{\Morita}$ has the two out of three property and is closed under retracts.
        \item\label{cond:modst2} $\mathcal{I}$ and $\mathcal{J}_{\Morita}$ admit the small object argument, where $\mathcal{I}$ denotes a generating set of cofibrations of $\sCat$.
        \item\label{cond:modst3} $\inj{\mathcal{J}_{\Morita}}\cap \mathcal{W}_{\Morita} = \inj{\mathcal{I}}$.
        \item\label{cond:modst4} $\cell{\mathcal{J}_{\Morita}}\subset  \mathcal{W}_{\Morita}$.
    \end{enumerate}
    Condition (\ref{cond:modst1}) is clearly satisfied, by Lemma~\ref{lem:Morita_r-equiv}. Condition (\ref{cond:modst2}) follows from the fact that the category of simplicial categories is locally presentable, and condition~(\ref{cond:modst3}) is Lemma~\ref{lem:tfib}, since the class $\inj{\mathcal{J}_{\Morita}}$ coincides with the class of Morita fibrations, by Remark~\ref{rem:rlp_interval}.

    The class of Morita weak equivalences is closed under transfinite composition. Therefore, to check condition (\ref{cond:modst4}) we only need to show that pushouts of maps in $\mathcal{J}_{\Morita}$ are Morita weak equivalences. All maps of the form $\Ug(\Lambda[n, k])\to \Ug(\Delta[n])$ for $n\ge 1$ and $0\le k\le n$ are trivial cofibrations in the fibered model structure, hence their pushouts are fibered weak equivalences and, in particular, Morita weak equivalences.

Suppose now that $h\colon E\to R$ is a retract interval in $\mathcal{G}$ and consider a pushout square in $\sCat$
    \begin{equation}\label{eq:posquare}
    \xymatrix{ E \ar[r]^-{a}
                 \ar[d]_-{h} &
               \mathcal{C} \ar[d]^-{k} \\
               R \ar[r]_-{b} &
               \mathcal{D}.
        }
    \end{equation}
The only object of $\mathcal{D}$ which is not in the image of $k$ is $b(1)$, which is a strong homotopy retract of $b(0)$ by construction. Hence $k$ is homotopically essentially surjective up to retracts.

It remains to check that $k$ is a local weak equivalence. The square (\ref{eq:posquare}) can be decomposed into two pushout squares
    \[
    \xymatrix{ E \ar[r]^-{a}
               \ar[d]_-{h^u} &
               \mathcal{C} \ar[d]^-{k'} \\
               h^*(R) \ar[r]^-{}
               \ar[d]_-{c} &
               \mathcal{C}' \ar[d]^-{k''} \\
               R \ar[r]_-{b} &
               \mathcal{D}.
    }
    \]
    The map of simplicial categories $c$ is fully faithful and injective on objects. It follows that $k''$ is fully faithful and injective on objects too, and hence, in particular, a local weak equivalence. In the top square the map $\trf{h}$ is a trivial cofibration in the fibered model structure, since $E(0,0)\to h^*(R)(0,0)$ is the same as the map $h_!E(0,0)=E(0,0)\to R(0,0)$ which is a cofibration by \cite[Theorem~7.13]{Mu14}. Therefore $k'$ is a fibered weak equivalence and, in particular, a local weak equivalence. It follows by the two out of three property for weak equivalences that $k$ is a local weak equivalence too.
\end{proof}

\begin{rem}
The Morita model structure for categories described in Theorem~\ref{thm:Morita_cat} can be also constructed by using retract intervals in $\Cat$. The proof is essentially the same as for simplicial categories. The only difference is that, since every object of $\Cat$ is fibrant, we can take a set of generating retract intervals consisting of just one object (cf. \cite[Lemma 2.1]{BM13}). Thus, for $\Cat_{\Morita}$ we can take $\iota\colon \Idem\to\Split$ as a sole generating trivial cofibration. The set of generating cofibrations is the same as the one for the canonical model structure. The fact that the model structure obtained this way coincides with the one of Theorem~\ref{thm:Morita_cat} is now easy to check using~\cite[Proposition E.1.10]{Joy08}, since they both have the same cofibrations and fibrant objects.
\end{rem}

\subsection{The Morita model structure as a localized model structure}\label{sect:loc_scat}

In this section we are going to show that the model structure $\sCat_{\Morita}$ can be obtained as a suitable localization of the Bergner model structure on simplicial categories. More precisely,
we will show that $\sCat_{\Morita}$ is the left Bousfield localization of $\sCat$ with respect to the map $\iota\colon \Idem\to \Split$. Since the cofibrations in $\sCat_{\Morita}$ and in $\sCat$ coincide, it is enough to prove that the $\iota$-local objects coincide with the fibrant objects of the Morita model structure.
\begin{propo}\label{prop:iota_fibrant}
Let $\mathcal{C}$ be a simplicial category. The following are equivalent:
\begin{itemize}
\item[{\rm (i)}] $\mathcal{C}$ is $\iota$-local.
\item[{\rm (ii)}] $\mathcal{C}$ is fibrant in the Morita model structure.
\item[{\rm (iii)}] $\mathcal{C}$ is locally fibrant and has the right lifting property with respect to $\St\hcN(\iota)$.
\end{itemize}
\end{propo}
\begin{proof}
We first prove that (i) and (ii) are equivalent. Note that a fibrant object in the Morita model structure is $\iota$-local because $\iota$ is a Morita weak equivalence (the Morita model structure is the left Bousfield localization of the Bergner model structure with respect to the class of Morita weak equivalences).

Conversely, let $\mathcal{C}$ be an $\iota$-local object and let $h$ be a retract interval. We have to prove that $\mathcal{C}$ has the right lifting property with respect to $h$. Note that $\mathcal{C}$ is $h$-local, since $h$ is weakly equivalent to $\iota$, for every retract interval $h$. Since $h$ is a cofibration between cofibrant objects in $\sCat$, it follows that $\mathcal{C}$ has the right lifting property with respect to $h$, by~\cite[Proposition 17.7.5(2)]{Hi03}.

To prove the equivalence between (ii) and (iii) observe first that $\St\hcN(\iota)$ is a retract interval, therefore (ii) implies (iii).

Conversely, we have to show that if $\mathcal{C}$ has the right lifting property with respect to $\St\hcN(\iota)$, then it has the right lifting property with respect to any retract interval~$h$. This follows by \cite[Proposition A.2.3.1]{Lu09}, since $h$ is weakly equivalent to $\St\hcN(\iota)$ and both are cofibrations between cofibrant objects.
\end{proof}

\begin{coro}\label{cor:morita_scat_loc}
The model structure $\sCat_{\Morita}$ is equal to the left Bousfield localization $L_{\iota}\sCat$ of the Bergner model structure on $\sCat$ with respect to the map $\iota\colon \Idem\to\Split$. In particular,  since $\sCat$ is left proper, $\sCat_{\Morita}$ is also left proper.
\end{coro}
\section{The Morita model structure for quasicategories}\label{sect:morita_qcat}
In this section, we introduce the Morita model structure for quasicategories. The existence of this model structure was first stated by Joyal in~\cite[Ch.16]{Joy08b}.

Recall that $\iota\colon \Idem\to\Split$ is the functor defined in Section~\ref{subsect:Morita_cat} and that $N\colon \Cat\to \ssets$ denotes the nerve functor.

\begin{defi}\label{def:morita_qcat}
The \emph{Morita model structure for quasicategories} $\ssets_{\Morita}$ is the left Bousfield localization of the Joyal model structure on simplicial sets with respect to the map $N(\iota)\colon N(\Idem)\to N(\Split)$.
\end{defi}

We give a characterization of the $N(\iota)$-local equivalences in Corollary~\ref{cor:char_moritaeq_ssets}. The fibrant objects of $\ssets_{\Morita}$, that is, the $N(\iota)$-local objects, can be characterized as follows:

\begin{propo}\label{propo:char_local_sset_morita}
A simplicial set $X$ is $N(\iota)$-local if and only if it is a quasicategory with the right lifting property with respect to $N(\iota)$.
\end{propo}
\begin{proof}
Suppose that $X$ is $N(\iota)$-local. Then $X$  is fibrant in the Joyal model structure, that is, $X$ is a quasicategory. The map $N(\iota)$ is a cofibration between cofibrant objects and hence $X$ has the right lifting property with respect to  $N(\iota)$, by \cite[Corollary 17.7.5(2)]{Hi03}.

The converse follows from \cite[Corollary 4.4.5.14]{Lu09}. (Lurie uses the notation $\Idem$ and $\Idem^+$ for what we call $N(\Idem)$ and $N(\Split)$, respectively, and he calls a quasicategory idempotent complete if it has the right lifting property with respect to $N(\iota)$.)
\end{proof}

The Quillen equivalence between simplicial sets with the Joyal model structure and simplicial categories with the Bergner model structure induces  a Quillen equivalence between the Morita model structures.
\begin{theo}\label{thm:Quillen_eq_ssets_morita}
The adjunction $\St:\ssets_{\Morita}\rightleftarrows\sCat_{\Morita}: \hcN$ is a Quillen equivalence.
\end{theo}
\begin{proof}
Recall that the homotopy coherent nerve and its left adjoint give a Quillen equivalence $\St:\ssets\rightleftarrows\sCat:\hcN$ between simplicial sets with the Joyal model structure and simplicial categories with the Bergner model structure. If we take the left Bousfield localization of the Joyal model structure with respect to $N(\iota)$, then using \cite[Theorem 3.3.20(1)(b)]{Hi03}, we obtain that there is a Quillen equivalence $\St:\ssets_{\Morita}\rightleftarrows L_{\St N(\iota)}\sCat: \hcN$. But the localized model structure $L_{\St N(\iota)}\sCat$ is the same as $\sCat_{\Morita}$, since $\St N(\iota)$ is weakly equivalent to $\iota$; see Remark~\ref{rem:CN_retract}.
\end{proof}

The functor $\St$ preserves and reflects all weak equivalences between the Morita model structures, since all objects in $\ssets_{\Morita}$ are cofibrant. This allows us to give a characterization of the weak equivalences in the Morita model structure for simplicial sets, which is the analogue of the characterization of the weak equivalences for the Morita model structure on simplicial categories (Theorem~\ref{thm:DK} and Lemma~\ref{lem:Morita_r-equiv}). Given a simplicial set $X$ it does not make sense to consider diagrams of simplicial sets indexed by $X$. Instead, we can consider simplicial functors from $\St (X)$ to simplicial sets. The category of such functors is Quillen equivalent, via straightening and unstraightening, to the slice category $\ssets/X$, as we now explain. 

Recall that a map of simplicial sets is called a \emph{right fibration} (respectively a \emph{left fibration}) if it has the right lifting property with respect to the horn inclusions $\Lambda[n,k]\to \Delta[n]$ for $0<k\le n$ (respectively for $0\le k<n$). For every simplicial set $X$ there is a model structure on the slice category $\ssets/X$ called the \emph{contravariant model structure} whose cofibrations are the monomorphisms and whose fibrant objects are the right fibrations. Similarly, there is also a model structure on $\ssets/X$ called the \emph{covariant model structure} whose cofibrations are the monomorphisms and whose fibrant objects are the left fibrations.

Given a simplicial set $X$ there is an adjunction
$$
r_!:\ssets/X\myrightleftarrows{\rule{0.4cm}{0cm}} \ssets^{\St(X)}:r^*
$$
given by the \emph{straightening} and \emph{unstraightening} functors; see \cite[Section 2.2.1]{Lu09}.
The following result can be deduced from~\cite[Theorem 2.2.1.2]{Lu09} (cf.~\cite[Proposition B]{HM16}).

\begin{theo}[Lurie, Heuts--Moerdijk]\label{thm:lurie_st-un}
For any simplicial set $X$, the adjunction $(r_!, r^*)$ is a Quillen equivalence between the slice category $\ssets/X$ with the covariant model structure, and the category $\ssets^{\St(X)}$ with the projective model structure.
\end{theo}
\begin{coro}\label{cor:char_moritaeq_ssets}
A map $f\colon X\to Y$ of simplicial sets is a weak equivalence in the Morita model structure $\ssets_{\Morita}$ if and only if it induces a Quillen equivalence $f_!:\ssets/X\rightleftarrows \ssets/Y:f^*$ between the slice categories with the covariant model structures.
\end{coro}
\begin{proof}
A map $f$ is a weak equivalence in the Morita model structure for simplicial sets if and only if $\St(f)$ is a weak equivalence in the Morita model structure for simplicial categories. For every map of simplicial sets $f\colon X\to Y$ the diagram
$$
\xymatrix{
\ssets/X \ar[d]_{f_!}\ar[r]^-{r_!} & \ssets^{\St(X)} \ar[d]^{\St(f)_!} \\
\ssets/Y \ar[r]_-{r_!} & \ssets^{\St(Y)}
}
$$
commutes up to natural isomorphism. The result follows from Theorem~\ref{thm:DK}, Lemma~\ref{lem:Morita_r-equiv} and Theorem~\ref{thm:lurie_st-un}.
\end{proof}

\begin{rem} Given a simplicial set $X$, we denote by $X^{\rm op}$ the opposite simplicial set, that is, the simplicial set obtained by composing the functor $X\colon \Delta^{\rm op}\to \sets$ with the automorphism of $\Delta$ which reverses the order relation on each ordinal. The result in~\cite[Theorem 2.2.1.2]{Lu09} establishes a Quillen equivalence between $\ssets/X^{\rm op}$ with the \emph{contravariant} model structure and $\ssets^{\St(X)}$ with the projective model structure. Since the functor $(-)^{\rm op}$ induces a Quillen equivalence between $\ssets/X^{\rm op}$ with the contravariant model structure and $\ssets/X$ with the covariant model structure,  Theorem~\ref{thm:lurie_st-un} immediately follows.

Also note that $\iota=\iota^{\rm op}$, so $f\colon X\to Y$ is a Morita weak equivalence of simplicial sets if and only if $f^{\rm op}\colon X^{\rm op}\to Y^{\rm op}$ is so. Therefore, we also have that $f\colon X\to Y$ is a weak equivalence in the Morita model structure $\ssets_{\Morita}$ if and only if it induces a Quillen equivalence $f_!:\ssets/X\rightleftarrows \ssets/Y:f^*$ between the slice categories with the \emph{contravariant} model structures.
\end{rem}

\section{The Morita model structure for operads}
In this section, we endow the category of operads in sets with the Morita model structure and, using the results of~\cite{CG18}, we characterize the weak equivalences in terms of equivalences of categories of algebras. We start by recalling the canonical model structure on operads, which extends the canonical model structure on small categories described in Section~\ref{subsect:can_cat}.

\subsection{The canonical model structure}
Let $\Fin$ be a skeleton for the category of finite sets, $\Sym$ the wide subcategory of $\Fin$ spanned by all the isomorphisms, and $C$ be a fixed set. We will denote by $\Sym_C$ the comma category $\Sym\downarrow C$. Its objects can be represented as finite sequences $\str{c}=(c_1,\ldots, c_n)$ in~$C$. Let $\Sign{C}$ denote the category $\Sym_C\times C$. The objects of this category will be written as $(\str{c}; c)$, where $\str{c}$ is an object of $\Sym_C$ and $c\in C$.

Let $\mathcal{O}$ be a $C$-coloured operad and let $\mathcal{P}$ be a $D$-coloured operad. A morphism $f\colon \mathcal{O}\to\mathcal{P}$ is called:
\begin{itemize}
\item[{\rm (i)}] \emph{Fully faithful} if the map $\morp{f}{\ope{O}(\str{c};c)}{\ope{P}(f(\str{c});f(c))}$ is an isomorphism for every $(\str{c}; c)$ in $\Sign{C}$.
\item[{\rm (ii)}] \emph{Essentially surjective} if for every $d\in D$ there exists a $c\in C$ such that $d$ is isomorphic to $f(c)$.
\end{itemize}

Recall that there is an adjunction between the category of small categories and the category of coloured operads
\begin{equation}\label{eq:adj_cat_oper}
j_!:\Cat\myrightleftarrows{\rule{0.4cm}{0cm}} \Oper :j^*.
\end{equation}
The right adjoint $j^*$ associates to every $C$\nobreakdash-coloured operad its underlying category with set of objects $C$. More explicitly, $j^*(\ope{O})(c,d)\cong \ope{O}(c;d)$ for every $c, d \in C$ and composition and identities are inherited directly from the ones in $\ope{O}$. Note that a morphism of operads $f$ is essentially surjective if and only if $j^*(f)$ is essentially surjective (in the sense of categories).

The category $\Oper$ of coloured operads admits a cofibrantly generated model structure, called the \emph{canonical model structure}, in which the weak equivalences are the operadic equivalences, that is, fully faithful and essentially surjective morphisms of operads; fibrations are the morphisms $f$ such that $j^*(f)$ is an isofibration; and cofibrations are the morphisms injective on colours.  The existence of this model structure  was shown in~\cite[Theorem 1.6.2]{Wei07} . It can also be deduced from the more general case for enriched operads~\cite[Theorem 4.22]{C14} or from the theory of model 2-categories~\cite[Theorem 4.3]{La07}.	

Given two operads $\mathcal{O}$ and $\mathcal{P}$, we can define a simplicial enrichment
$$
\Map(\mathcal{O},\mathcal{P})=N(\Iso( j^*\Fun(\mathcal{O},\mathcal{P}))),
$$
where $\Fun(\mathcal{O},\mathcal{P})$ is the operad of morphisms from $\mathcal{O}$ to $\mathcal{P}$.
With this simplicial enrichment, the canonical model structure on $\Oper$ is a simplicial model structure (like the canonical model structure on $\Cat$; see Section~\ref{subsect:can_cat}). This can be seen by using that the canonical model structure on $\Oper$ is a $\Cat$-model structure, by~\cite[Theorem 4.3]{La07}, and the fact that $N(\Iso(-))$ preserves fibrations and trivial fibrations.

\subsection{The Morita model structure}

Let $\ope{O}$ be a symmetric coloured operad in $\sets$ with set of colours~$C$. Given $c,c'\in C$ the colour \emph{$c$ is a retract of $c'$} if there exist two operations $r\in \ope{O}(c';c)$ and $i \in \ope{O}(c;c')$ such that $ri=\id_c$.

\begin{defi}
    Let $\ope{O}$ and $\ope{P}$ be two coloured operads in $\sets$, with set of colours $C$ and $D$, respectively. A morphism of coloured operads $\morp{f}{\ope{O}}{\ope{P}}$ is called:
    \begin{itemize}
        \item[{\rm (i)}] \emph{Essentially surjective up to retracts} if for every $d\in D$ there exists a $c\in C$ such that $d$ is isomorphic to a retract of $f(c)$.
	\item[{\rm (ii)}] A \emph{Morita equivalence} if it is fully faithful and essentially surjective up to retracts.
        \end{itemize}
\end{defi}

As above, a morphism of operads $f$ is essentially surjective up to retracts if and only if $j^*(f)$ is essentially
surjective up to retracts (in the sense of categories).

Morita equivalences of operads can be characterized similarly as for categories. The \emph{Cauchy completion} $\overline{\mathcal{O}}$ of an operad $\mathcal{O}$ can be defined as follows: the colours are pairs $(c, e)$, where $c\in C$ and $e\in \mathcal{O}(c;c)$ is an idempotent operation.
An operation in $\overline{\mathcal{O}}((c_1, e_1), \ldots, (c_n, e_n); (c, e))$ is an operation $p\in \mathcal{O}(c_1,\ldots, c_n; c)$ such that $p\circ (e_1,\ldots, e_n)=p=e\circ p$.  The canonical morphism $\mathcal{O}\to\overline{\mathcal{O}}$ that sends $c$ to $(c, \id_c)$ is a Morita equivalence.

The following statements are equivalent for a morphism $f\colon\mathcal{O}\to \mathcal{P}$ of operads:
\begin{itemize}
\item[{\rm (i)}] $f$ is a Morita equivalence.
\item[{\rm (ii)}] $\overline{f}\colon\overline{\mathcal{O}}\to\overline{\mathcal{P}}$ is an equivalence of operads.
\end{itemize}

\begin{theo}\label{thm:oper_morita}
There is a cofibrantly generated model structure $\Oper_{\Morita}$ on the category of coloured operads in $\sets$ in which the weak equivalences are the Morita equivalences and the cofibrations are the morphisms that are injective on objects. The fibrant objects are the Cauchy complete operads.
\end{theo}
\begin{proof}
The model structure $\Oper_{\Morita}$ is the left Bousfield localization of the canonical model structure with respect to the morphism $j_!(\iota)\colon j_!(\Idem)\to j_!(\Split)$. To identify the weak equivalences and the fibrant objects one proceeds as in the proof of Theorem~\ref{thm:Morita_cat}.
\end{proof}

\begin{rem}
Let $\eta$ denote the terminal category seen as an operad. Then, under the identification $\Cat=\Oper/\eta$, we can recover the Morita model structure $\Cat_{\Morita}$ of  Theorem~\ref{thm:Morita_cat} from the standard model structure on the slice category $\Oper/\eta$ induced by that of Theorem~\ref{thm:oper_morita}.
\end{rem}

The Morita equivalences between operads in $\sets$ can be characterized as those morphisms of coloured operads which induce an equivalence of categories between the respective categories of algebras. The following result is proved in \cite[Theorem 4.5]{CG18}, by 
means of algebraic theories.

\begin{theo}\label{thm:at_ess_surj}
Let $\morp{f}{\ope{O}}{\ope{P}}$ be a morphism of symmetric coloured operads in $\sets$. The following are equivalent:
    \begin{itemize}
        \item[{\rm (i)}] The morphism $f$ is a Morita equivalence of operads.
        \item[{\rm (ii)}] The induced adjunction
        \[
        f_!:\Alg(\ope{O})\myrightleftarrows{\rule{0.4cm}{0cm}} \Alg(\ope{P}): f^*
        \]
        is an equivalence of categories.
    \end{itemize}
\end{theo}

\section{The Morita model structure for simplicial operads}

Let $\Mcc=\ssets$ and denote the category of symmetric coloured $\Mcc$-operads by $\sOper$. Recall that there is a functor $\pi_0\colon\sOper\to \Oper$ that sends a simplicial $C$\nobreakdash-coloured operad $\mathcal{O}$ to the $C$-coloured operad in sets $\pi_0(\mathcal{O})$, whose set of morphisms is defined by $\pi_0(\mathcal{O})(\str{c}; c)=\pi_0(\mathcal{O}(\str{c};c))$, for every $(\str{c}; c)$ in $\Sign{C}$.

\subsection{The Cisinski--Moerdijk model structure}
Cisinski and Moerdijk proved that there exists a cofibrantly generated model structure on $\sOper$ that extends the Bergner model structure on $\sCat$ and models the homotopy theory of $\infty$\nobreakdash-operads; see~\cite[Theorem 1.14]{CM13}. We call this model structure the
\emph{Cisinski--Moerdijk model structure} on simplicial operads and we denote it by $\sOper_{\CM}$.

The weak equivalences in the Cisinski\nobreakdash--Moerdijk model structure, that we call \emph{operadic weak equivalences}, are the operadic analogues of the Dwyer--Kan equivalences of simplicial categories. They are explicitly defined as follows:

\begin{defi}
Let $\mathcal{O}$ be a $C$-coloured operad. A morphism $\morp{f}{\ope{O}}{\ope{P}}$ in $\sOper$ is called:
    \begin{itemize}
        \item[{\rm (i)}] \emph{Homotopically fully faithful} (respectively a \emph{local fibration}) if the map
        \[
        \morp{f}{\ope{O}(\str{c};c)}{\ope{P}(f(\str{c});f(c))}
        \]
        is a weak equivalence (respectively a fibration) of simplicial sets, for every $(\str{c},c)$ in $\Sign{C}$.
        \item[{\rm (ii)}] An \emph{operadic weak equivalence} if it is homotopically fully faithful and $\pi_0(f)$ is an equivalence of operads.
    \end{itemize}
\end{defi}

\begin{theo}[Cisinski--Moerdijk]
There is a right proper cofibrantly generated model structure on the category of simplicial operads $\sOper$, called the Cisinski--Moerdijk model structure, in which weak equivalences are the operadic weak equivalences and fibrations are the maps that are local fibrations and isofibrations.
\end{theo}

The adjunction~(\ref{eq:adj_cat_oper}) extends to an adjunction between the categories of simplicial categories and simplicial coloured operads
\begin{equation}
\label{eq:adj_scat_soper}
j_!:\sCat\myrightleftarrows{\rule{0.4cm}{0cm}} \sOper :j^*.
\end{equation}
Moreover, $(j_!,j^*)$ is a Quillen pair between the Bergner model structure on $\sCat$ and the
Cisinski\nobreakdash--Moerdijk model structure on $\sOper$.

\subsection{The Morita model structure} In this section, we prove the existence of the Morita model structure for simplicial operads. We begin by describing the weak equivalences and fibrations of this model structure.

\begin{defi}
A morphism $f$ in $\sOper$ is called:
    \begin{itemize}
        \item[{\rm (i)}] \emph{Homotopically essentially surjective up to retracts} if $\pi_0(f)$ is essentially surjective up to retracts.
        \item[{\rm (ii)}] A \emph{Morita weak equivalence} if it is homotopically fully faithful and homotopically essentially surjective up to retracts.
    \end{itemize}
\end{defi}

The aim of this section is to prove that there exists a model structure on the category of simplicial coloured operads in which the class of weak equivalences is the class of Morita weak equivalences. Its fibrations are the natural generalization of the Morita fibrations of simplicial categories.

\begin{defi}
    Let $\ope{O}$ be a $C$\nobreakdash-coloured operad. A morphism $\morp{f}{\ope{O}}{\ope{P}}$ of coloured operads in $\sOper$ is called:
    \begin{itemize}
        \item[{\rm (i)}] \emph{Retract-lifting} if $j^*(f)$ is retract-lifting (see Definition~\ref{def:ret_lift}(i)).
        \item[{\rm (ii)}] A \emph{Morita fibration} if it is a local fibration and retract-lifting.
    \end{itemize}
\end{defi}

Observe that, given a coloured operad $\ope{O}$ and two colours $x,y$ in $C$, we have that $y$ is a (homotopical) retract of $x$ if and only if $y$ is a (homotopical) retract of $x$ in~$j^*(\ope{O})$.
\begin{lemma} \label{lemma:moritatrifib}
    A map of simplicial coloured operads is a Morita fibration and a Morita weak equivalence if and only if it is a local trivial fibration surjective on objects.
\end{lemma}
\begin{proof} It is enough to prove that if $f$ is a local fibration; then $j^*(f)$ is retract-lifting and essentially surjective up to retracts if and only if $j^*(f)$ is surjective on objects. This follows from Lemma~\ref{lem:tfib}.
\end{proof}

For every simplicial set $X$ and every $n\in \N$, let $C_n[X]$ be the unique simplicial coloured operad with set of colours $\{0,1,\dots,n\}$ such that
\[
C_n[X](\str{c};c)=
\begin{cases}
X & \text{ if }(\str{c};c)=(1,\dots,n;0),\\
\ast & \text{ if } \str{c} = c,\\
\emptyset & \text{otherwise.}
\end{cases}
\]
The assignment $C_n[-]$ is clearly functorial in $X$ for every $n\in\mathbb{N}$.
Consider the set of morphisms of simplicial operads
\[
\mathcal{J}_{{\rm loc}}=\{C_n[\Lambda[m,k]] \longrightarrow C_n[\Delta[m]] \mid  n\ge 0,\,m\ge 1,\, 0\le k\le m\}.
\]
The following characterization of the morphisms that are local fibrations can be found in \cite[Lemma 1.16]{CM13}:
\begin{propo}
A map of simplicial operads  is a local fibration if and only if it has the right lifting property with respect to $\mathcal{J}_{{\rm loc}}$.
\end{propo}

The proof of the following theorem is very close to the one of Theorem \ref{theo:moritamodelcat}; see also the proof of \cite[Theorem 1.14]{CM13}.
\begin{theo}\label{thm:Morita_simp_oper}
    There exists a cofibrantly generated model structure on $\sOper$, called the Morita model structure, and
    denoted by $\sOper_{\Morita}$, such that:
\begin{itemize}
        \item[{\rm (i)}] The weak equivalences are the Morita weak equivalences.
        \item[{\rm (ii)}] The fibrations are the Morita fibrations.
        \item[{\rm (iii)}] The cofibrations are the cofibrations of $\sOper$ with the Cisinski--Moerdijk model structure.
\end{itemize}
A set of generating cofibrations consists of the set of generating cofibrations for the Cisinski\nobreakdash--Moerdijk model structure.  As set of generating trivial cofibrations we can take the set $\mathcal{J}_{\Morita}=\mathcal{J}_{{\rm loc}}\cup j_!(\mathcal{G})$, where $\mathcal{G}$ is any generating set of retract intervals.
\end{theo}
\begin{proof}
    As in the proof of Theorem \ref{theo:moritamodelcat}, we have to check the four conditions of Kan's recognition principle for cofibrantly generated model categories:
\begin{enumerate}
        \item\label{cond:modopst1} The class of Morita weak equivalences $\mathcal{W}_{\Morita}$ has the two out of three property and is closed under retracts.
        \item\label{cond:modopst2} $\mathcal{I}$ and $\mathcal{J}_{\Morita}$ admit the small object argument, where $\mathcal{I}$ denotes a generating set of cofibrations of $\sOper_{\CM}$.
        \item\label{cond:modopst3} $\inj{\mathcal{J}_{\Morita}}\cap \mathcal{W}_{\Morita} = \inj{\mathcal{I}}$.
        \item\label{cond:modopst4} $\cell{\mathcal{J}_{\Morita}}\subset  \mathcal{W}_{\Morita}$.
    \end{enumerate}

    Condition~(\ref{cond:modopst1}) is easily checked, and condition~(\ref{cond:modopst2}) readily follows from the fact that $\sOper$ is locally presentable (see~\cite[Corollary 2.9.2.6]{Cav_PhD}). Lemma~\ref{lemma:moritatrifib} guarantees that condition~(\ref{cond:modopst3}) holds.
    The only requirement left to check for Kan's recognition principle is condition~(\ref{cond:modopst4}), this is, that the class $\cell{\mathcal{J}_{\Morita}}$ is contained in the class of Morita weak equivalences.

    Since Morita weak equivalences are closed under transfinite composition and all the maps in $\mathcal{J}_{{\rm loc}}$ are operadic weak equivalences, it is enough to check that for every retract interval of simplicial categories $\morp{h}{E}{R}$ and every pushout diagram
    \[
    \xymatrix{
        j_!(E) \ar[d]_-{j_!(h)} \ar[r]^-{f} & \mathcal{O} \ar[d]^-{k}\\
        j_!(R) \ar[r]_-{f'} & \mathcal{P},
    }
    \]
    the map of operads $k$ is a Morita weak equivalence. The only colour of $\mathcal{P}$ which is not in the image of $k$ is $f'(1)$, which is clearly a homotopy retract of $f'(0)=k(f(0))$. This proves that $k$ is homotopically essentially surjective up to retracts, hence we only have to prove that it is homotopically fully faithful.

    Factoring $h$ into a bijective on objects functor $\trf{h}$ followed by a fully faithful and injective on objects functor $c$, the above diagram decomposes into two pushouts
    \[
    \xymatrix{
        j_!(E) \ar[d]_-{j_!(\trf{h})} \ar[r]^-{f} & \mathcal{O} \ar[d]^-{k'}\\
        j_!(h^*(R)) \ar[d]_-{j_!(c)} \ar[r]^-{f} & \mathcal{O'} \ar[d]^-{k''}\\
        j_!(R) \ar[r]_-{f'} & \mathcal{P}.
    }
    \]
    The morphism $j_!(\trf{h})$ is a trivial cofibration in $\sOper_{\CM}$, since $j_!$ is a left Quillen functor. Therefore $k'$ is homotopically fully faithful.
    By \cite[Lemma 1.29]{CM13} the morphism $k''$ is homotopically fully faithful. It follows that $k'\circ k''=k$ is homotopically fully faithful.
\end{proof}

\begin{propo}\label{pro:morita_soper_loc}
The model structure $\sOper_{\Morita}$ is equal to the left Bousfield localization of $\sOper_{\CM}$ with respect to $j_!(\iota)$. A simplicial operad is Morita fibrant if and only if it is locally fibrant and it has the right lifting property with respect to $j_!\St hcN(\iota)$.
\end{propo}
\begin{proof}
Since a model structure is completely determined by its cofibrations and its fibrant objects (see~\cite[Proposition E.1.10]{Joy08}), it is enough to show that the $j_!(\iota)$\nobreakdash-local objects coincide with the fibrant objects of $\sOper_{\Morita}$. Every fibrant object in $\sOper_{\Morita}$ is $j_!(\iota)$\nobreakdash-local, since $j_!(\iota)$ is a Morita equivalence of operads. Conversely, suppose that $\mathcal{O}$ is $j_!(\iota)$\nobreakdash-local. Then $\mathcal{O}$ is fibrant in $\sOper_{\Morita}$ if and only if $j^*(\mathcal{O})$ has the right lifting property with respect to every retract
interval $h$. This follows from Proposition~\ref{prop:iota_fibrant}, since if $\mathcal{O}$ is $j_!(\iota)$-local, then $j^*(\mathcal{O})$ is $\iota$-local.

To characterize the Morita fibrant simplicial operads, note that,  by Proposition~\ref{prop:iota_fibrant}, $j^*(\mathcal{O})$ has the right lifting property with respect to every retract interval if and only if it has the right lifting property with respect to $\St hcN(\iota)$. By the adjunction~(\ref{eq:adj_scat_soper}), this is the same as requiring that $\mathcal{O}$ has the right lifting property with respect to $j_!\St hcN(\iota)$.
\end{proof}

\begin{rem}
Observe that the left Bousfield localization of $\sOper_{\CM}$ with respect to $j_!(\iota)$ cannot be inferred from the usual existence theorems (\cite[Theorem~4.1.1]{Hi03}, \cite[Theorem 4.7]{Bar10}), since they require the model structure to be left proper, which does not hold for the Cisinski--Moerdijk model structure on $\sOper$; see~\cite[Section 4]{HRY}.
\end{rem}

Let $\mathcal{O}$ be a $C$-coloured operad in simplicial sets and let $\Alg({\mathcal{O}})$ denote the corresponding category of algebras. An $\mathcal{O}$-algebra $X=(X(c))_{c\in C}$ is an object of $\ssets^C$ together with a morphism of operads $\mathcal{O}\to {\rm End}(X)$, where ${\rm End}(X)$ denote the $C$-coloured operad of endomorphisms of $X$.

The category $\Alg({\mathcal{O}})$ admits a transferred model structure via the free-forgetful adjunction, which is called the \emph{projective model structure}; see \cite[Theorem~2.1]{BM07}. The weak equivalences and the fibrations are defined colourwise, that is, they are the morphisms $X\to Y$ such that $X(c)\to Y(c)$ is a weak equivalence or a fibration of simplicial sets for every $c\in C$, respectively. Moreover, for a morphism of operads $f\colon \mathcal{O}\to\mathcal{P}$ the
induced adjunction
        \[
        f_!:\Alg(\ope{O})\myrightleftarrows{\rule{0.4cm}{0cm}} \Alg(\ope{P}): f^*
        \]
is a Quillen pair with respect to the corresponding projective model structures.

Recall that a symmetric $C$-coloured operad $\ope{O}$ is called \emph{$\Sigma$-cofibrant} if it has an underlying cofibrant $C$-coloured collection, that is, if for every~$(\str{c}, c)$ in $\Sign{C}$ the simplicial set $\ope{O}(\str{c};c)$ is cofibrant in $\ssets^{\Sym_C(\str{c},\str{c})}$ endowed with the projective model structure.

The following characterization of Morita weak equivalences of simplicial operads in terms of their categories of algebras is proved in \cite[Theorem 4.8]{CG18}, by means of simplicial algebraic theories.

\begin{theo}\label{theo:main morita oper}
Let  $\morp{f}{\ope{O}}{\ope{P}}$ be a map between $\Sigma$-cofibrant simplicial coloured operads. The following are equivalent:
    \begin{enumerate}
        \item[{\rm (i)}] The map $f$ is a Morita weak equivalence of simplicial operads.
        \item[{\rm (ii)}] The Quillen pair
        \[
        f_!:\Alg(\ope{O})\myrightleftarrows{\rule{0.4cm}{0cm}} \Alg(\ope{P}): f^*
        \]
        is a Quillen equivalence.
    \end{enumerate}
\end{theo}

\begin{rem}
    The hypothesis of $\Sigma$-cofibrancy in the previous theorem might seem restrictive and technical. However, it is helpful to keep in mind the following two facts: every cofibrant operad in the Cisinski--Moerdijk model structure is $\Sigma$\nobreakdash-cofibrant and the cofibrant resolution of every operad provides a model for the corresponding notion of homotopy invariant algebraic structure; see \cite{BM07}.

Hence, the above theorem can be  read as follows: a morphism of simplicial coloured operads is a Morita weak equivalence if and only if it induces a Quillen equivalence between the corresponding categories of homotopy invariant algebraic structures. Therefore, the Morita model structure on $\sOper$ provides a model for a homotopy theory of homotopy invariant algebraic structures.
\end{rem}

\section{The Morita model structure for dendroidal sets}

We begin this section by recalling some generalities about dendroidal sets. For complete details, we refer the reader to~\cite{MW07} and \cite{Moe10}. 

Let $\Omega$ denote the category of trees introduced by Moerdijk--Weiss in~\cite{MW07} as an extension of the simplicial category $\Delta$. The objects of $\Omega$ are (non-planar) rooted trees. Every tree $T$ in $\Omega$ has an associated symmetric coloured  operad $\Omega(T)$, whose set of colours is the set of edges of $T$ and whose operations are generated by the vertices of $T$. More explicitly, $\Omega(T)$ is the free symmetric coloured operad on the coloured collection determined by $T$. The set of morphisms between two trees $S$ and $T$ is defined as
$$
\Omega(S, T)=\Oper(\Omega(S), \Omega(T)).
$$
The category  $\dsets$  of \emph{dendroidal sets} is the category of presheaves on $\Omega$. Given a tree $T$, we denote by $\Omega[T]$ the representable presheaf $\Omega(-, T)$.

Similarly as in $\Delta$ every morphism in $\Omega$ can be factored in a unique way as a composition of faces, followed by an isomorphism and followed by a compostion of degeneracies. Given a face map $\alpha$ of $T$, the \emph{$\alpha$-horn} $\Lambda^{\alpha}[T]$ is the dendroidal subset of $\Omega[T]$ consisting of the union of the images of all faces of $\Omega[T]$ except $\alpha$ (if $\alpha$ is a face map that contracts an inner edge, $\Lambda^{\alpha}[T]$ is called an \emph{inner horn}).  The inclusion $\Lambda^{\alpha}[T]\to \Omega[T]$ is called a horn inclusion. 

A dendroidal set $X$ is called an \emph{$\infty$-operad} if it has the right lifting property with respect to all inner horn inclusions.

A monomorphism $f\colon X \to Y$ of dendroidal sets is called a \emph{normal monomorphism} if the action of the automorphism group $\Aut(T)$  on $Y_T\setminus f(X_T)$ is free. A dendroidal set $X$ is called \emph{normal} if the unique map $\emptyset\to X$ is a normal monomorphism.

There is a fully faithful inclusion $i\colon \Delta\to \Omega$, where $[n]$ is sent to the linear tree with $n$ vertices and $n+1$ edges, which induces an adjunction
\begin{equation}\label{eq:adj_dsets_ssets}
i_!:\ssets\myrightleftarrows{\rule{0.4cm}{0cm}} \dsets: i^*.
\end{equation}
There is a \emph{dendroidal nerve} functor $N_d$ from operads to dendroidal sets, defined as $N_d(\mathcal{O})_T=\Oper(\Omega(T),\mathcal{O})$. The functor $N_d$ is fully faithful, has a left adjoint and it extends the simplicial nerve of categories $N$, that is, the following diagrams commute up to natural isomorphism:
$$
\xymatrix{
\Oper \ar[r]^-{j^*} \ar[d]_-{N_d} & \Cat \ar[d]^-N \\
\dsets \ar[r]_-{i^*} & \ssets,
} \qquad
\xymatrix{
\Oper  \ar[d]_-{N_d} & \Cat \ar[l]_-{j_!} \ar[d]^-N \\
\dsets  & \ar[l]^-{i_!} \ssets.
}
$$

The category of dendroidal sets admits a model structure that generalizes the Joyal model structure on simplicial sets~\cite[Theorem 2.4, Proposition 2.6]{CM11}.

\begin{theo}[Cisinski--Moerdijk]
The category of dendroidal sets admits a left proper combinatorial model structure, called the operadic model structure, whose cofibrations are the normal monomorphisms, the fibrant objects are the $\infty$-operads, and whose weak equivalences are the operadic weak equivalences.
\end{theo}
The adjoint pair~(\ref{eq:adj_dsets_ssets}) is a Quillen pair  between the Joyal model structure on simplicial sets and the operadic model structure on dendroidal sets.

The category of dendroidal sets is closely related to the category of simplicial operads, extending the relation of simplicial sets with simplicial categories via the homotopy coherent nerve. If $\mathcal{O}$ is a simplicial operad, we denote by $W(\mathcal{O})$ its Boardman--Vogt construction with respect to the interval $\Delta[1]$. The operad $W(\mathcal{O})$ is a cofibrant replacement of $\mathcal{O}$ and the $W(\mathcal{O})$-algebras are the $\mathcal{O}$-algebras up to homotopy; see~\cite[Section~4]{CM13} and \cite[Section 3]{BM07} for details. If we apply the $W$-construction to the operads $\Omega(T)$ viewed as discrete simplicial operads we get a functor $\Omega\to \sOper$ sending a tree $T$ to $W(\Omega(T)$ which, by left Kan extension, gives a pair of adjoint functors
\begin{equation}\label{eq:adj_hcN_d}
\St_d:\dsets\myrightleftarrows{\rule{0.4cm}{0cm}} \sOper: \hcN_d,
\end{equation}
extending adjunction $\eqref{eq:adj_hcN}$. The right adjoint is called the \emph{homotopy coherent dendroidal nerve} and, given a simplicial operad $\mathcal{O}$, it is defined by
$$
\hcN_d(\mathcal{O})_T=\sOper(W(\Omega(T)), \mathcal{O}).
$$
Cisinski and Moerdijk proved in~\cite[Theorem 8.15]{CM13} that adjunction \eqref{eq:adj_hcN_d} is a Quillen equivalence.

\begin{theo}[Cisinski--Moerdijk]
The adjunction $\St_d:\dsets\rightleftarrows\sOper:\hcN_d$ is a Quillen equivalence, where $\dsets$ is endowed with the operadic model structure and $\sOper$ is endowed with the Cisinski--Moerdijk model structure.
\end{theo}

We now describe the Morita model structure for dendroidal sets, which can be defined directly as a left Bousfield localization of the operadic model structure on dendroidal sets, since the latter is left proper and combinatorial.

\begin{defi}\label{def:morita_dsets}
The \emph{Morita model structure} for dendroidal sets $\dsets_{\Morita}$ is the left Bousfield localization of the operadic  model structure with respect to $i_! N(\iota)$ (or, equivalently, with respect to $N_dj_!(\iota)$).
\end{defi}

The Morita weak equivalences of dendroidal sets are characterized in Corollary~\ref{cor:char_moritaweq_dsets}. The fibrant objects of $\dsets_{\Morita}$, that is, the $i_!N(\iota)$-local objects, can be characterized as follows:

\begin{propo}
A dendroidal set is fibrant in $\dsets_{\Morita}$ if and only if it is an $\infty$-operad and has the right lifting property with respect to
$i_!N(\iota)$.
\end{propo}
\begin{proof}
By \cite[Proposition 17.4.16]{Hi03} if $X$ is fibrant in the Cisinski--Moerdijk model structure, then $X$ is $i_!N(\iota)$-local if and only if $i^*X$ is $N(\iota)$-local. The proof now reduces to the case of quasicategories, which was treated in Proposition~\ref{propo:char_local_sset_morita}.
\end{proof}

The Quillen equivalence between dendroidal sets with the operadic model structure and simplicial operads with the Cisinski--Moerdijk model structure induces  a Quillen equivalence between the Morita model structures.
\begin{theo}\label{thm:Quillen_eq_dsets_morita}
The adjunction $\St_d:\dsets_{\Morita}\rightleftarrows\sOper_{\Morita}: \hcN_d$ is a Quillen equivalence.
\end{theo}
\begin{proof}
The model structure $\sOper_{\Morita}$ is the same as the localized model structure $L_{j_!(\iota)}\sOper$. Since $i_! N\simeq N_d j_!$, $\St_d i_!\simeq j_!\St$ and $\St N(\iota)\simeq \iota$, where $\simeq$ means weakly equivalent, we have that $\St_d N_dj_!(\iota)\simeq j_!(\iota)$; see Remark~\ref{rem:CN_retract}. Therefore, the localized model structure $L_{\St_d N_dj_!(\iota)}\sOper$ exists and is the same as $\sOper_{\Morita}$. If we consider the left Bousfield localization of the operadic model structure on dendroidal sets with respect to $N_dj_!(\iota)$, then by \cite[Theorem 3.3.20(1)(b)]{Hi03}, there is a Quillen equivalence $\St_d:\dsets_{\Morita}\rightleftarrows L_{\St_d N_dj_!(\iota)}\sOper: \hcN_d$ and the result follows.
\end{proof}

\begin{rem}
Let $\eta$ denote the representable dendroidal set corresponding to the tree with one edge and no vertices, and also the terminal simplicial category seen as an operad. The functors $\St_d$  and $\hcN_d$ preserve $\eta$. Thus, under the identifications $\ssets=\dsets/\eta$ and $\sCat=\sOper/\eta$, we can recover the Morita model structures on quasicategories (Definition~\ref{def:morita_qcat}) and simplicial categories (Theorem~\ref{theo:moritamodelcat}) as the standard model structure on the corresponding slice categories induced by the Morita model structures for dendroidal sets (Definition~\ref{def:morita_dsets}) and simplicial operads (Theorem~\ref{thm:Morita_simp_oper}), respectively.

By the same argument, the Quillen equivalence of Theorem~\ref{thm:Quillen_eq_ssets_morita} between the Morita model structure for quasicategories and the Morita model structure for simplicial categories can be recovered, by slicing over $\eta$, from the Quillen equivalence of Theorem~\ref{thm:Quillen_eq_dsets_morita}.
\end{rem}

The functor $\St_d$ preserves and reflects all weak equivalences between normal dendroidal sets. This allows us to give a characterization of the weak equivalences between normal dendroidal sets in the Morita model structure, similarly as we did in Section~\ref{sect:morita_qcat} for quasicategories, as we now explain.

Given a simplicial operad $\mathcal{O}$, we denote by $\Alg(\mathcal{O})$  the category of $\mathcal{O}$-algebras in simplicial sets equipped with the projective model structure. Generalizing the covariant model structure on simplicial sets described by Lurie, Heuts proved in \cite[Theorem 2.3]{Heu} that given a dendroidal set $X$ the slice category $\dsets/X$ admits the \emph{covariant model structure}, whose cofibrations are the normal monomorphisms and whose fibrant objects are the left fibrations $A\to X$. Moreover, given a dendroidal set $X$ there is an adjunction
$$
r_!:\dsets/X\myrightleftarrows{\rule{0.4cm}{0cm}} \Alg({\St_d(X))}:r^*
$$
given by the \emph{straightening} and \emph{unstraightening} functors; see \cite[Section 2.2]{Heu}.
The following result can be found in~\cite[Theorem 2.7]{Heu} (cf.~\cite[Corollary~ 6.5]{BoaM}).
\begin{theo}[Heuts, Boavida de Brito--Moerdijk]\label{thm:Heuts}
For any normal dendroidal set $X$ the adjunction $(r_!, r^*)$ is a Quillen equivalence between the slice category $\dsets/X$ with the covariant model structure, and the category $\Alg(\St_d(X))$ with the projective model structure.
\end{theo}

\begin{coro}\label{cor:char_moritaweq_dsets}
A map $f\colon X\to Y$ between normal dendroidal sets is a weak equivalence in the Morita model structure $\dsets_{\Morita}$ if and only if it induces a Quillen equivalence $f_!:\dsets/X\rightleftarrows \dsets/Y: f^*$ between the slice categories with the covariant model structures.
\end{coro}
\begin{proof}
A map $f$ between normal dendroidal sets is a weak equivalence in the Morita model structure for simplicial sets if and only if $\St_d(f)$ is a weak equivalence in the Morita model structure for simplicial operads. For every map of normal dendroidal sets $f\colon X\to Y$ the diagram
$$
\xymatrix{
\dsets/X \ar[d]_{f_!}\ar[r]^-{r_!} & \Alg({\St(X)}) \ar[d]^{\St_d(f)_!} \\
\dsets/Y \ar[r]_-{r_!} & \Alg({\St_d(Y)})
}
$$
commutes up to natural isomorphism. The result follows from Theorem~\ref{theo:main morita oper} and  Theorem~\ref{thm:Heuts}.
\end{proof}

\end{document}